\documentclass[11pt]{article}
\usepackage{a4wide}
\usepackage {times,epic,eepic,graphicx}
\usepackage {amsthm,amsmath,amssymb,amsfonts,amscd}
\usepackage {dsfont}
\usepackage {alltt}
\usepackage[utf8]{inputenc}
\usepackage[english]{babel}
\usepackage[english]{isodate}
\usepackage[parfill]{parskip}

\usepackage {graphicx,epsfig,psfrag}
\usepackage {algorithm}
\usepackage {algorithmic}
\usepackage {graphicx}
\usepackage {multirow}
\usepackage [figuresright]{rotating}
\usepackage {rotating}
\usepackage {subfigure}
\usepackage {epsfig}
\usepackage {url}
\usepackage{todonotes}

    \renewcommand  {\phi}    { \varphi }

    \newcommand    {\ten}    { \otimes }
    \newcommand    {\Ten}    { \bigotimes }

    \newcommand    {\bq}   { \begin{equation} }
    \newcommand    {\eq}   { \end{equation} }
    \newcommand    {\ba}   { \begin{array} }
    \newcommand    {\ea}   { \end{array} }
    \newcommand    {\BbbN}  {\mathds{N}}
    \newcommand    {\N}     {\BbbN}

    \newcommand    {\BbbR}  {\mathds{R}}
    \newcommand    {\R}     {\BbbR}

    \newcommand    {\Id}    { \mbox{\bf Id} }
    
    \newcommand    {\id}    { \mbox{id} }

    \newcommand    {\Bild}  { \mbox{\rm range}}
    \newcommand    {\Rang}  { \mbox{\rm rank}}
    
    \newcommand    {\Kern}  { \mbox{\rm kernel} }

    \newcommand    {\Diag}  { \mbox{\rm diag} }

    \newcommand    {\Span}  { \mbox{\rm span} }
    \newcommand    {\argmin}{ \mbox{\rm argmin} }
    \newcommand    {\argmax}{ \mbox{\rm argmax} }

\newcommand    {\dotprod}[2] { \left \langle #1, #2 \right \rangle }

\newcommand    {\ve}[1]
                {\underline{#1}}

\newcommand    {\dist}[2]
                {\mbox{\rm dist} \left(#1,#2\right) }

\newcommand    {\Ap}
                {\mathcal{A}}

\newcommand    {\Ip}
                {\mathcal{I}}

\newcommand    {\Kp}
                {\mathcal{K}}

\newcommand    {\Vp}
                {\mathcal{V}}

\newcommand{\BIGOP}[1]{\mathop{\mathchoice%
{\raise-0.22em\hbox{\huge $#1$}} {\raise-0.05em\hbox{\Large $#1$}}
{\hbox{\large $#1$}}{#1}}}

\newcommand{\BIGboxplus}{\mathop{\mathchoice%
{\raise-0.35em\hbox{\huge $\boxplus$}}%
{\raise-0.15em\hbox{\Large $\boxplus$}}{\hbox{\large
$\boxplus$}}{\boxplus}}}
\newcommand{\bigtimes}{\BIGOP{\times}}

\newtheorem{theorem}{Theorem}[section]
\newtheorem{remark}[theorem]{Remark}
\newtheorem{defn}[theorem]{Definition}
\newtheorem{note}[theorem]{Notation}

\newtheorem{assumption}[theorem]{Assumption}

\newtheorem{example}[theorem]{Example}

\newtheorem{lemma}[theorem]{Lemma}

\newtheorem{corollary}[theorem]{Corollary}

\newtheorem{proposition}[theorem]{Proposition}

\newlength{\symboheight}
\graphicspath{{images/}}
\ifx\undefined\pdfpageheight
  \def\image#1#2{\epsfig{file=#1,width=#2}}
\else
  \def\image#1#2{\epsfig{file=#1,width=#2}}
\fi

\begin{document}
\title{On the Convergence of Alternating Least Squares Optimisation in Tensor Format Representations}
\author{
  Mike Espig
     \thanks{RWTH Aachen University, Germany}
       \footnote{
     Address: RWTH Aachen University, Department of Mathematics, IGPM
Pontdriesch 14-16, 52062 Aachen Germany.
     Phone: +49 (0)241 80 96343, E-mail address: mike.espig@alopax.de
   }
  \and
  Wolfgang Hackbusch
 \thanks{Max Planck Institute for Mathematics in the Sciences, Leipzig, Germany}
  \and
  Aram Khachatryan
  \footnotemark[1]
  }
\maketitle

\begin{abstract}

The approximation of tensors is important for the efficient
numerical treatment of high dimensional problems, but it remains an
extremely challenging task. One of the most popular approach to
tensor approximation is the alternating least squares method. In our
study, the convergence of the alternating least squares algorithm is
considered. The analysis is done for arbitrary tensor format
representations and based on the multiliearity of the tensor format.
In tensor format representation techniques, tensors are approximated
by multilinear combinations of objects lower dimensionality. The
resulting reduction of dimensionality not only reduces the amount of
required storage but also the computational effort.

\end{abstract}

{\bf Keywords}:  tensor format, tensor representation, tensor
network, alternating least squares optimisation, orthogonal
projection method.

{\bf MSC}:  15A69, 49M20, 65K05, 68W25, 90C26.


\section{Introduction}\label{sec:introduction}
During the last years, tensor format representation techniques were
successfully applied to the solution of high-dimensional problems
like stochastic and parametric partial differential equations
\cite{Doostan12, ESHALIMA11_1, Waehnert2012, KhoromskijSchwab2011,
MatthiesZander2012, Nouy07, Nouy10}. With standard techniques it is
impossible to store all entries of the discretised high-dimensional
objects explicitly. The reason is that the computational complexity
and the storage cost are growing exponentially with the number of
dimensions. Besides of the storage one should also solve this
high-dimensional problems in a reasonable (e.g. linear) time and
obtain a solution in some compressed
 (low-rank/sparse) tensor formats. Among other prominent problems, the efficient solving of
linear systems is one of the most important tasks in scientific computing.\\
We consider a minimisation problem on the tensor space
$\Vp=\Ten_{\nu=1}^d \R^{m_\nu}$ equipped with the Euclidean inner
product $\dotprod{\cdot}{\cdot}$. The objective function $f: \Vp \rightarrow \R$
of the optimisation task is quadratic
\begin{equation}\label{equ:deff}
f(v):=\frac{1}{\|b\|^2}\left[\frac{1}{2} \dotprod{Av}{v} -
\dotprod{b}{v} \right],
\end{equation}
where $A \in \R^{m_1 \cdots m_d \times m_1 \cdots m_d}$ is a
positive definite matrix ($A>0,\, A^T=A$) and $b \in \Vp$. A tensor
$u\in \Vp$ is represented in a tensor format. A tensor format $U:
P_1\times \dots \times P_L \rightarrow \Vp$ is a multilinear map
from the cartesian product of parameter spaces $P_1, \dots, P_L$
into the tensor space $\Vp$. A $L$-tuple of vectors $(p_1, \dots,
p_L)\in P:=P_1 \times \dots \times P_L$ is called a representation
system of $u$ if $u=U(p_1, \dots, p_L)$. The precise definition of
tensor format representations is given in Section
\ref{sec:tensorFormat}. The solution $A^{-1}b=\argmin_{v\in\Vp}f(v)$
is approximated by elements from the range set of the tensor format
$U$, i.e. we are looking for a representation system $(p^*_1, \dots,
p^*_L)\in P$ such that for
\begin{eqnarray}\label{equ:defF}
    F&:=&f \circ U :P \rightarrow \Vp \rightarrow \R\\
    \nonumber
    F(p_1, \dots, p_L)&=&\frac{1}{\|b\|^2}\left[\frac{1}{2} \dotprod{A U(p_1, \dots, p_L)}{U(p_1, \dots, p_L)} -
\dotprod{b}{U(p_1, \dots, p_L)} \right]
\end{eqnarray}
we have
\begin{equation*}
F(p^*_1, \dots, p^*_L) = \inf_{(p_1, \dots, p_L) \in P} F(p_1,
\dots, p_L).
\end{equation*}
The alternating least squares (ALS) algorithm \cite{BEMO02, BEMO05,
ESHAHARS11_2, HoltzALS2012, Kolda09tensordecompositions,
Oseledets2011, OseledetsDolgov2012} is iteratively defined. Suppose
that the $k$-th iterate $\ve{p}^k=(p_1^k, \dots, p_L^k)$ and the
first $\mu-1$ components $p_1^{k+1}, \dots, p_{\mu-1}^{k+1}$ of the
$(k+1)$-th iterate $\ve{p}^{k+1}$ have been determined. The basic
step of the ALS algorithm is to compute the minimum norm solution
\begin{equation*}
    p_\mu^{k+1}:=\argmin_{q_\mu \in P_\mu}F(p_1^{k+1}, \dots, p_{\mu-1}^{k+1}, q_\mu, p_{\mu+1}^{k}, \dots,
    p_{L}^{k}).
\end{equation*}
Thus, in order to obtain $\ve{p}^{k+1}$ from $\ve{p}^k$, we have to
solve successively $L$ ordinary least squares problems.
The ALS algorithm is a nonlinear Gauss--Seidel method. The local
convergence of the nonlinear Gauss--Seidel method to a stationary
point $\ve{p}^* \in P$ follows from the convergence of the linear
Gauss--Seidel method applied to the Hessian $F''(\ve {p}^*)$ at the
limit point $\ve {p}^*$. If the linear Gauss--Seidel method
converges R-linear then there exists a neighbourhood $B(\ve {p}^*)$
of $\ve {p}^*$ such that for every initial guess $\ve{p}^0 \in B(\ve
{p}^*)$ the nonlinear Gauss--Seidel method converges R-linear with
the same rate as the linear Gauss--Seidel method. We refer the
reader to Ortega and Rheinboldt for a description of nonlinear
Gauss--Seidel method \cite[Section 7.4]{OR70} and convergence
analysis \cite[Thm. 10.3.5, Thm. 10.3.4, and Thm. 10.1.3]{OR70}. A
representation system of a represented tensor is not unique, since
the tensor representation $U$ is multilinear. Consequently, the
matrix $F''(\ve {p}^*)$ is not positive definite. Therefore,
convergence of the linear Gauss--Seidel method is in general not
ensured. However, if the Hessian matrix at $\ve{p}^*$ is positive
semidefinite then the linear Gauss--Seidel method still converges
for sequences orthogonal to the kernel of $F''(\ve {p}^*)$, see e.g.
\cite{Keller1965, Lee:2006}. Under useful assumptions on the null
space of $F''(\ve {p}^*)$, Uschmajew et al. \cite{UschmajewALS2013,
UschmajewALS2012} showed local convergence of the ALS method. These
assumptions are related to the nonuniqueness of a representation
system and meaningful in the context of a nonlinear Gauss Seidel
method. However, for tensor format representations the assumptions
are not true in general, see the counterexample of Mohlenkamp
\cite[Section 2.5]{Mohlenkamp2013}
and discussion in \cite[Section 3.4]{UschmajewALS2012}.\\
The current analysis is not based on the mathematical techniques
developed for the nonlinear Gauss--Seidel method, but on the
multilinearity of the tensor representation $U$. This fact is in
 contrast to previous works. The present article is partially related to
 the study by Mohlenkamp \cite{Mohlenkamp2013}. For example, the statement of Lemma
\ref{lemma:dist} is already described for the canonical tensor
format.\\
Section \ref{sec:tensorFormat} contains a unified mathematical
description of tensor formats. The relation between an orthogonal
projection method and the ALS algorithm is explained in Section
\ref{sec:ALS}. The convergence of the ALS method is analysed in
Section \ref{sec:Analyse}, where we consider global convergence.
Further, the rate of convergence is described in detail and explicit
examples for all kind of convergent rates are given. The ALS method
can converge for all tensor formats of practical interest
sublinearly, Q-linearly, and even Q-superlinearly\footnote{We refer
the reader to \cite{OR70} for details concerning convergence
speed.}. We illustrate our theoretical results on  numerical
examples in Section \ref{sec:numericalExamples}.

\section{Unified Description of Tensor Format Representations}\label{sec:tensorFormat}
A tensor format representation for tensors in $\Vp$ is described by
a parameter space $P=\bigtimes_{\mu=1}^{L} P_\mu$ and a multilinear
map $U:P \rightarrow \Vp$ from the parameter space into the tensor
space. For the numerical treatment of high dimensional problems by
means of tensor formats it is essential to distinguish between a
tensor $u \in \Vp$ and a representation system $\ve {p} \in P$ of
$u$, where $u=U(\ve {p})$. The data size of a representation system
is often proportional to $d$. Thanks to the multilinearity of $U$,
the numerical cost of standard operations like matrix vector
multiplication, addition, and computation of scalar products is also
proportional to $d$, see e.g. \cite{ESHAHARS11_2, GR10, HAK09, OS11,
OSTY09}.

\begin{note}[$\N_{n}$]
The set $\N_n$ of natural numbers smaller than $n \in \N$ is denoted
by
\begin{equation*}
    \N_n:=\{j \in \N : 1\leq j \leq n\}.
\end{equation*}
\end{note}

\begin{defn}[Parameter Space, Tensor Format Representation, Representation System]\label{defn:IncidentMapTFTFR}
Let $L \geq d$, $\mu \in \N_d$, and $P_\mu$ a finite dimensional
vector spaces equipped with an inner product $\dotprod{\cdot}{\cdot}_{P_\mu}$.
The \emph{parameter space} $P$ is the following cartesian product
\begin{equation}\label{equ:paprameterSpace}
    P:= \bigtimes_{\mu=1}^{L} P_\mu.
\end{equation}
A multilinear map $U$ from the parameter space $P$ into the tensor
space $\Vp$ is called a \emph{tensor format representation}
\begin{equation}\label{equ:tensorFormat}
    U : \bigtimes_{\mu=1}^{L} P_\mu \rightarrow \Ten_{\nu=1}^d \R^{m_\nu}.
\end{equation}
We say $u \in \Vp$ is represented in the tensor format
representation $U$ if $u \in \Bild U$. A tuple $(p_1, \dots, p_L)
\in P$ is called a \emph{representation system} of $u$ if $u=U(p_1,
\dots, p_L)$.
\end{defn}

\begin{remark}
Due to the multilinearity of $U$, a representation system of a given tensor $u \in \Bild(U)$ is not uniquely
determined.
\end{remark}

\begin{example}\label{exa:tensorFormats}

For the canonical tensor format representation with $r$-terms we
have $L=d$ and $P_\mu = \R^{m_\mu \times r}$. The canonical tensor
format representation with $r$-terms is the following multilinear
map
\begin{eqnarray*}
  U_{CF} : \bigtimes_{\mu=1}^d \R^{m_\mu \times r} &\rightarrow & \Vp\\
   (p_1, \dots, p_d)\mapsto U_{CF}(p_1, \dots, p_d)&:=&\sum_{j=1}^r \Ten_{\mu=1}^d
   p_{\mu,j},
\end{eqnarray*}
where $p_{\mu,j}$ denotes the $j$-th column of the matrix $p_\mu \in
\R^{m_\mu \times r}$. For recent algorithms in the canonical tensor
format we refer to
\cite{ES08, ESHAGA09, ESHA09_1, ESHALIMA11_1, ESHAROSCH09_1}.\\

The tensor train (TT) format representation discussed in \cite{OS11}
is for $d=3$ and representation ranks $r_1, r_2 \in \N$ defined by
the multilinear map
\begin{eqnarray*}
  U_{TT}& :& \R^{m_1 \times r_1} \times \R^{m_2 \times r_1 \times r_2} \times \R^{m_3 \times r_2}   \rightarrow  \R^{m_1} \ten \R^{m_2} \ten \R^{m_3}\\
(p_1, p_2, p_3)&\mapsto& U_{TT}(p_1, p_2, p_3):=\sum_{i=1}^{r_1}
\sum_{j=1}^{r_2} p_{1, i} \ten p_{2, i, j} \ten p_{3, j}.
\end{eqnarray*}
\end{example}%

\section{Orthogonal Projection Method and Alternating Least Squares Algorithm}\label{sec:ALS}

It is shown in the following that the ALS algorithm is an orthogonal
projection method on subspaces of $\Vp=\Ten_{\nu=1}^d \R^{m_\nu}$.
For a better understanding, we briefly repeat the description of
projection methods, see e.g. \cite{brezinski1997projection, saad2000iterative} for a detailed
description.

An orthogonal projection method for solving the linear system $A v =
b$ is defined by means of a sequence $(\Kp_k)_{k\in \N}$  of
subspaces of $\Vp$ and the construction of a sequence
$(v_k)_{k\in\N} \subset \Vp$ such that
\begin{equation*}
    v_{k+1}\in \Kp_k \quad \mbox{and} \quad r_{k+1} = b-A v_{k+1} \perp \Kp_k.
\end{equation*}
A prototype of projection method is explained in Algorithm
\ref{alg:prototypeProjectionMethod}.

\begin{algorithm}[h]\caption{Prototype Projection Method}\label{alg:prototypeProjectionMethod}
\begin{algorithmic}[1]
    \WHILE{Stop Condition}
        \STATE  Compute an orthonormal basis $V_k=\left[u^k_{1}, \dots,
        u^k_{m_k}\right]$ of $\Kp_k$
        \STATE $r_k := b - Av_k$
        \STATE $v_{k+1} = v_k + V_k (V_k^T A V_k)^{-1} V_k^T r_k$
    \STATE $k \mapsto k +1$
    \ENDWHILE
  \end{algorithmic}
\end{algorithm}
\begin{note}[$L(A,B)$]
Let $A, B$ be two arbitrary vector spaces. The vector space of
linear maps from $A$ to $B$ is denoted by
\begin{equation*}
 L(A,B):=\left\{ \phi : A \rightarrow B : \phi \mbox{ is linear}\right\}.
\end{equation*}
\end{note}
In the following, let $U : P \rightarrow \Vp$ be a tensor format
representation, see Definition \ref{defn:IncidentMapTFTFR}. We need
to define subspaces of $\Vp$ in order to show that the ALS algorithm
is an orthogonal projection method. The multilinearity of $U$ and
the special form of the ALS micro-step are important for the
definition of these subspaces. Let $\mu \in \N_L$ and $v \in \Vp$ be
a tensor represented in the tensor format $U$, i.e. there is $(p_1,
\dots, p_{\mu-1}, p_\mu,p_{\mu+1}, \dots, p_L) \in P$ such that
$v=U(p_1, \dots, p_{\mu-1}, p_\mu,p_{\mu+1}, \dots, p_L)$. Since the
tensor format representation $U$ is multilinear we can define a
linear map $W_\mu(p_1, \dots, p_{\mu-1}, p_{\mu+1}, \dots, p_L) \in
L(P_\mu, \Vp)$ such that $v=W_\mu(p_1, \dots, p_{\mu-1}, p_{\mu+1},
\dots, p_L)p_\mu$. The map $W_\mu$ depends multilinearly on the
parameter $p_1, \dots, p_{\mu-1}, p_{\mu+1}, \dots, p_L$. The linear
subspace $\Bild\left(W_\mu(p_1, \dots, p_{\mu-1}, p_{\mu+1}, \dots,
p_L) \right) \subseteq \Vp$ is of great importance for the ALS
method. For the rest of the article, we identify linear maps with
its canonical matrix representation.
\begin{defn}\label{def:W_mu}
    Let $\mu \in \N_L$. We write for a given representation system $\ve{p}=(p_1, \dots, p_L) \in P$
\begin{equation*}
    \ve{p}^{[\mu]}:= (p_1, \dots, p_{\mu-1}, p_{\mu+1}, \dots p_L)
\end{equation*}
and define
\begin{eqnarray}\label{equ:Wmu}
    W_{\mu, \, \ve{p}^{[\mu]}} : P_\mu &\rightarrow& \Vp \\
    \nonumber
    \tilde{p}_\mu \mapsto W_{\mu, \, \ve{p}^{[\mu]}}\tilde{p}_\mu&:=&U(p_1, \dots, p_{\mu-1}, \tilde{p}_\mu ,p_{\mu+1}, \dots
    p_L).
\end{eqnarray}
We simply write $W_{\mu}$ for $W_{\mu, \,
\ve{p}^{[\mu]}}$, i.e. $W_{\mu}:=W_{\mu, \, \ve{p}^{[\mu]}}$ if it is clear from the context which representation system is
considered.
\end{defn}
\begin{proposition}\label{pro:OrthbasisandRepWmu}
Let $\mu \in \N_L$ and $\ve{p}= (p_1, \dots, p_L) \in P$. The
following holds:

\begin{itemize}
  \item[(i)] $W_{\mu, \, \ve{p}^{[\mu]}}$ is a linear map and $\Bild\left(
W_{\mu, \, \ve{p}^{[\mu]}}\right)$ is a linear subspace of $\Vp$.
  \item[(ii)] We have $\Rang\left(W_{\mu, \, \ve{p}^{[\mu]}}\right) \leq
  \dim(P_\mu)$.

  \item[(iii)] $\Bild \left(W_{\mu, \, \ve{p}^{[\mu]}}\right) \subset
  \Bild(U)$, i.e. for all $v \in \Bild \left(W_{\mu,\, \ve{p}^{[\mu]}}\right)$ there exist $\tilde{p_\mu} \in P_\mu$ such that
\begin{equation*}
    v = U(p_1, \dots, p_{\mu-1}, \tilde{p_\mu}, p_{\mu+1}, \dots,
    p_L).
\end{equation*}
  \item[(iv)] Set $H_\mu:=W_{\mu, \, \ve{p}^{[\mu]}}^T W_{\mu, \, \ve{p}^{[\mu]}} \in \R^{\dim P_\mu \times \dim P_\mu}$ and let $H_\mu = \tilde{U}_\mu D_\mu
\tilde{U}^T_\mu$ be the diagonalisation of the square matrix
$H_\mu$, where $D_\mu=\Diag(\delta_{i,\mu})_{i=1, \dots,  \dim
P_\mu}$ with $\delta_{1, \mu} \geq \delta_{2, \mu} \geq \dots \geq \delta_{\dim P_\mu, \mu}$.
Define further $\tilde{D_\mu}=\Diag(\tilde{\delta}_{i,
\mu})_{i=1, \dots, \Rang{(W_{\mu, \, \ve{p}^{[\mu]}})}}$. Then the
columns of
\begin{equation}\label{equ:defVmu}
    V_{\mu}:= W_{\mu, \, \ve{p}^{[\mu]}} \tilde{U}_\mu \tilde{D}_\mu^{-\frac{1}{2}}
\end{equation}
form an orthonormal basis of $\Bild \left(W_{\mu, \, \ve{p}^{[\mu]}}
\right)$ and
\begin{equation}\label{equ:ortEqW}
    V_{\mu}p_{\mu}= W_{\mu, \, \ve{p}^{[\mu]}}  \left(\tilde{U}_\mu
    \tilde{D}_\mu^{-\frac{1}{2}}p_\mu\right)=U(p_1, \dots, p_{\mu-1}, \tilde{U}_\mu
    \tilde{D}_\mu^{-\frac{1}{2}}p_\mu, p_{\mu+1}, \dots,
    p_L).
\end{equation}

\item[(v)] The map
\begin{eqnarray*}
    W_{\mu} : P_1 \times \dots \times P_{\mu-1} \times P_{\mu+1} \dots, P_L &\rightarrow& L(P_\mu,\Vp) \\
    \nonumber
    \tilde{p}_\mu \mapsto W_{\mu, \, \ve{p}^{[\mu]}}
\end{eqnarray*}
is multilinear.
\end{itemize}
\end{proposition}
\begin{proof}
Note that $W_{\mu, \, \ve{p}^{[\mu]}}$ is linear, since the tensor
format $U$ is multilinear. The rest of the assertions follows after
short calculations, where the last assertion (v) is a direct
consequence of the multilinearity of $U$.
\end{proof}

\begin{remark}
In chemistry the definition of $V_\mu$ in Proposition
\ref{pro:OrthbasisandRepWmu} (iv) is often called L\"{o}wdin
transformation, see \cite[Section 3.4.5]{szaOst1996}. Nevertheless,
the construction can be found in several proofs for the existence of
the singular value decomposition, see e.g. \cite[Lemma 2.19]{HA12}.
\end{remark}
\begin{defn}\label{defn:Fmu}
Let $\mu\in \N_L$, $\ve{p}=(p_1, \dots, p_L) \in P$, and $F:P
\rightarrow \R$ as defined in Eq. (\ref{equ:defF}). We define
 \begin{eqnarray}\label{equ.Fmu}
  F_{\mu, \ve{p}^{[\mu]}} : P_\mu
  &\rightarrow& \R\\ \nonumber
  \tilde{p}_\mu \mapsto F_{\mu, \ve{p}^{[\mu]}}(\tilde{p}_\mu)&:= &F(p_1, \dots, p_{\mu-1}, \tilde{p}_\mu , p_{\mu+1}, \dots,
  p_L).
 \end{eqnarray}
We write for convenience $F_{\mu}:=F_{\mu,
\ve{p}^{[\mu]}}$ if it is clear from the context which representation system is
considered.
\end{defn}
\begin{lemma}\label{lemma:p_muNew}
Let $\mu \in \N_L$ and $\ve{p}=(p_1, \dots, p_L) \in P$. We have

\begin{itemize}
  \item [(i)] $F'_\mu(q_\mu)=-W^T_\mu(b-AW_\mu q_\mu)$,
  \item [(ii)] $(V^T_\mu A V_\mu)^{-1}V^T_\mu b = \argmin_{q_\mu \in P_\mu}
    F_{\mu}(q_\mu)$,
\end{itemize}
where $V_\mu$ is defined in Eq. (\ref{equ:defVmu}).
\end{lemma}
\begin{proof} (i): Let $q_\mu \in P_\mu$. We have $f(W_\mu q_\mu)=F_\mu(q_\mu)$ for all $\mu \in
\N_L$ and
\begin{eqnarray*}
   F_\mu(q_\mu)=\frac{1}{\|b\|^2} \left[ \frac{1}{2} \dotprod{A W_\mu q_\mu}{W_\mu q_\mu} - \dotprod{b}{W_\mu q_\mu}
\right]=\frac{1}{\|b\|^2} \left[ \frac{1}{2} \dotprod{W_\mu^T A
W_\mu q_\mu}{q_\mu} - \dotprod{W_\mu^T b}{q_\mu} \right].
\end{eqnarray*}
Since $W_\mu^T A W_\mu$ is symmetric, we have $F'_\mu(q_\mu)=W_\mu^T
A W_\mu q_\mu - W_\mu^T b = -W^T_\mu(b-AW_\mu q_\mu)$.\\
(ii): For $V_\mu q_\mu \in \Bild(W_\mu)$ we can write
\begin{equation*}
    f(V_\mu q_\mu ) = \frac{1}{\|b\|^2} \left[ \frac{1}{2} \dotprod{V_\mu^T A
V_\mu q_\mu}{q_\mu} - \dotprod{V_\mu^T b}{q_\mu} \right].
\end{equation*}
Since $V_\mu$ is a basis of $\Bild(W_\mu)$, we have that $V_\mu^T A
V_\mu$ is positive definite and therefore
\begin{equation*}
    p^*_\mu=\argmin_{q_\mu \in P_\mu}
    F_{\mu}(q_\mu) \Leftrightarrow V_\mu^T A V_\mu p_\mu^* -
    V_\mu^Tb=0 \Leftrightarrow p_\mu^*=(V_\mu^T A V_\mu)^{-1}V_\mu^Tb.
\end{equation*}
\end{proof}
\begin{theorem}\label{theo:FmuPorj}
Let $\mu \in \N_L$ and $\ve{p}=(p_1, \dots, p_L) \in P$. We have
\begin{equation*}
    p_\mu^*=\argmin_{q_\mu \in P_\mu}
    F_{\mu}(q_\mu)  \quad \Leftrightarrow \quad b - A \,V_\mu p_\mu^* \perp \Bild W_\mu,
\end{equation*}
where $V_\mu$ is from Eq. (\ref{equ:defVmu}).
\end{theorem}
\begin{proof}
Follows from Lemma \ref{lemma:p_muNew} and orthogonal projection
theorem.
\end{proof}
\begin{algorithm}[h]\caption{Alternating Least Squares (ALS)
Algorithm}\label{alg:ALS}
\begin{algorithmic}[1]
    \STATE Set $k:=1$ and choose an initial guess $\ve{p}_1=(p^1_1, \dots, p_L^1)\in
    P$, $\ve{p}_{1,\, 0}:=\ve{p}_1$, and $v_{1}:=U(\ve{p}_1)$.
    \WHILE{Stop Condition}
        \STATE $v_{k, \,0}:=v_{k}$
        \FOR {${1 \leq \mu \leq L}$}
        \STATE  Compute an orthonormal basis $V_{k, \mu }$ of the range space of $W_{k, \mu}:=W_{\mu, \ve{p}_{k, \,
        \mu}^{[\mu]}}$, see e.g. Eq. (\ref{equ:Wmu}) and (\ref{equ:defVmu}).
        \STATE
            \begin{eqnarray}
                p_{\mu}^{k+1}&:=& \underbrace{\tilde{U}_{k, \, \mu} \tilde{D}_{k, \,
            \mu}^{-\frac{1}{2}} (V_{k, \mu}^T A V_{k, \mu})^{-1} \tilde{D}_{k, \,
            \mu}^{-\frac{1}{2}} \tilde{U}_{k, \mu}^T}_{G^+_{k, \, \mu}:=} W^T_{\mu}(p_1^{k+1}, \dots, p^{k+1}_{\mu-1},
p^{k}_{\mu+1}, \dots, p^k_L)b\label{eq:defPmu}\\ \nonumber
                \ve{p}_{k, \mu+1}&:=& (p_1^{k+1}, \dots, \,p_{\mu-1}^{k+1} \,, p_\mu^{k+1}, p_{\mu+1}^{k}, \dots,
                p_L^{k})\\ \nonumber
                r_{k, \mu } &:=& b - A v_{k, \, \mu}\\ \nonumber
                \Rightarrow v_{k, \mu+1} &=& v_{k, \mu} + V_{k, \mu}(V_{k, \mu}^T A V_{k, \mu})^{-1} V_{k, \mu}^T r_{k, \mu} = V_{k, \mu}(V_{k, \mu}^T A V_{k, \mu})^{-1} V_{k, \mu}^T b=U(\ve{p}_{k, \mu+1})
            \end{eqnarray}
            where $\tilde{U}_{k, \, \mu} \tilde{D}_{k, \,
            \mu}^{-\frac{1}{2}}$ is from Eq. (\ref{equ:defVmu})
        \ENDFOR
    \STATE $\ve{p}_{k+1}:= \ve{p}_{k, L}$ and
    $v_{k+1}:=U(\ve{p}_{k+1})$
    \STATE $k \mapsto k + 1$
    \ENDWHILE
  \end{algorithmic}
\end{algorithm}
\begin{figure}[h]
  \centering
  {\image{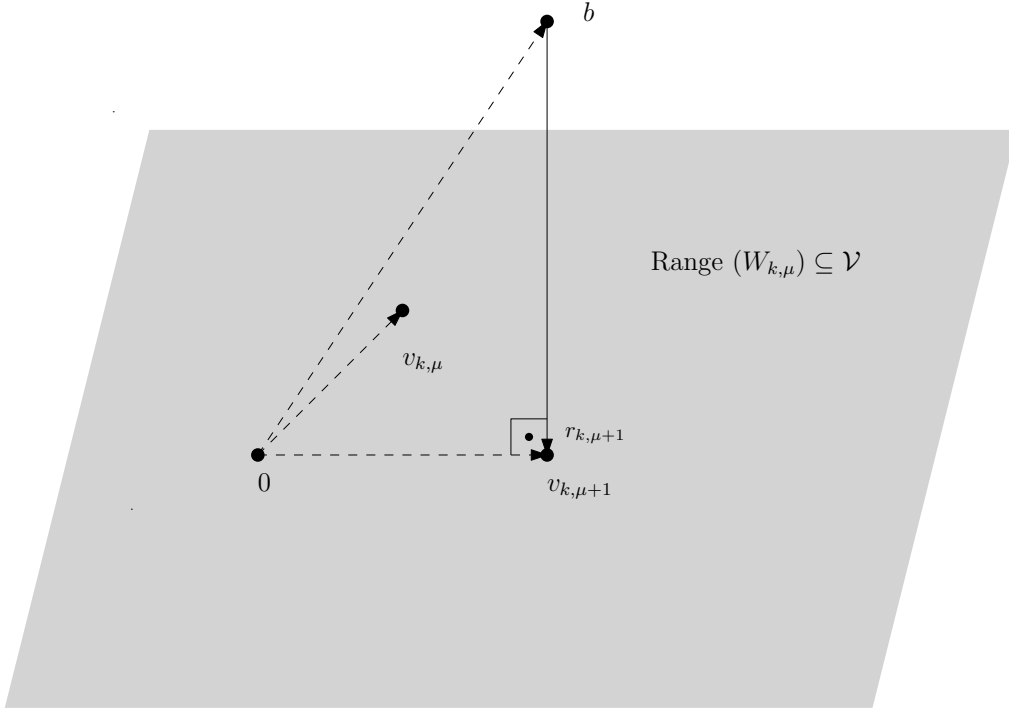}{0.8\textwidth}}
  \caption{Graphical illustration of an ALS micro-step for the case when $A=\id$. At the current iteration step, we define the linear map
  $W_{k, \mu} \in L(P_\mu, \Vp)$ by means of $v_{k, \mu}$ and the multilinearity of U, cf. Definition \ref{def:W_mu}. The successor $v_{k, \mu+1}$ is then the best approximation of $b$
  on the subspace $\Bild\left(W_{k, \mu}\right) \subseteq \Vp$.}
  \label{bild:projectionAls}
\end{figure}
\begin{remark}\label{remark:pOrthKern}
From the definition of $p_\mu^{k+1}$, it follows directly that
$p_\mu^{k+1} \perp \Kern(W_{\mu,k})$ and $p_\mu^{k+1}$ is the vector
\emph{with smallest norm} that fulfils the normal equation
$G_{k,\mu}\, p_\mu^{k+1} = W_{k,\mu}^Tb$. This is very important for
the convergence analysis of the ALS method and we like to point out
that our results are based on this condition. We must give special
attention in a correct implementation of an ALS micro-step in order
to fulfil this essential property.
\end{remark}

\section{Convergence Analysis}\label{sec:Analyse}

We consider global convergence of the ALS method. The convergence
analysis for an arbitrary tensor format representation $U:
\bigtimes_{\mu=1}^L P_\mu \rightarrow \Vp$ is a quite challenging
task. The objective function $F$ from Eq. (\ref{equ:defF}) is highly
nonlinear. Even the existence of a minimum is in general not
ensured, see \cite{DELI06} and \cite{landsberg12}. We need further
assumptions on the sequence from the ALS method. In order to justify
our assumptions, let us study an example from Lim and de~Silva
\cite{DELI06} where it is shown that the tensor
\begin{equation*}
    b = x \ten x \ten y + x \ten y \ten x + y \ten x \ten x
\end{equation*}
with tensor rank $3$ has no best tensor rank $2$ approximation. Lim
and de~Silva explained this by constructing a sequence $(v_k)_{k \in
\N}$ of rank $2$ tensors with
\begin{equation*}
    v_k = \left(x+\frac{1}{k}\,y\right) \ten
    \left(x+\frac{1}{k}\,y\right) \ten \left(k x+y\right) - x \ten x
    \ten kx \xrightarrow[k \rightarrow \infty]{} b.
\end{equation*}
The linear map $W_{1, k}$ from Definition \ref{def:W_mu} and the
first component vector $p_{1, k }$ of the parameter system have the following form:

\begin{eqnarray*}
 W_{1, k} &=& \left[\underbrace{\left( x + \frac{1}{k}y\right) \ten \left( x + \frac{1}{k}y\right) , \, x \ten
    x}_{\mbox{column vectors of the matrix}}\right]
    \ten \Id_{\R^n},\\
    p_{1}^k &=& \left(
                   \begin{array}{c}
                     1 \\
                     0 \\
                   \end{array}
                 \right) \ten (k x + y) + \left(
                   \begin{array}{c}
                     0 \\
                     1 \\
                   \end{array}
                 \right) \ten k x.
\end{eqnarray*}

It is easy to verify that the equation $W_{1,k} p_{1}^k = v_k$ holds. Furthermore, we have
\begin{eqnarray*}
  \lim_{k \rightarrow \infty} \|p_{1}^k\| &=& \infty,\\
  W_1=\lim_{k \rightarrow \infty} W_{1, k} &=& \left(
                                                 \begin{array}{cc}
                                                   x\ten x & x\ten x \\
                                                 \end{array}
                                               \right)\ten
                                               \Id_{\R^n}.
\end{eqnarray*}
Obviously, the rank of $W_1$  is equal to $n$ but $\Rang(W_{1,
k})=2n$ for all $k \in \N$. This example shows already that we need
assumptions on the boundedness of the parameter system and on the
dimension of the subspace $\Span (W_{\mu,k})$.

\begin{defn}[Critical Points]
The set $\mathfrak{M}$ of \emph{critical points} is defined by
\begin{equation}\label{eq:criticalPoint}
    \mathfrak{M}:=\left\{ v \in \Vp \,: \, \exists \ve{p} \in P : v=U(\ve{p}) \wedge F'(\ve{p})=0\right\}.
\end{equation}
\end{defn}
In our context, critical points are tensors that can be
represented in our tensor format $U$ and there exists a parameter
system $\ve{p}$ such that $(f \circ U)'(\ve{p}) =0$, i.e. $\ve{p}$
is a stationary point of $F=f \circ U$. A representation system of a
tensor $v=U(\ve{p})$ is never uniquely defined since the tensor
format is a multilinear map. The following remark shows that the
non uniqueness of a parameter system has even more subtle effects, in particular when the parameter system of $v=U(\ve{p})$ is also a stationary
point of $F$.

\begin{remark}
In general, $v \in \mathfrak{M}$ does not imply $F^\prime(\hat{p}) =
0$ for any parameter system $\hat{p}$ of $v$, i.e. there exist a tensor format $\tilde{U}$ and two
different  $p,\, \hat{p} \in P$ such that
$\tilde{U}(p)=\tilde{U}(\hat{p})$
 and $0=\tilde{F}'(p)\neq\tilde{F}'(\hat{p})$.
\end{remark}
\begin{proof}
Let

\begin{eqnarray*}
  \tilde{U} : \R^2 \times \R^2 &\rightarrow& \R^2 \ten \R^2 \simeq \R^4\\
    (x, y) &\mapsto& \tilde{U}(x, y) := \left(
                \begin{array}{c}
                  x_1 y_1 +x_2 y_1 \\ x_1 y_1 +x_2 y_1 \\ x_1y_2 \\ x_2y_2 \\
                \end{array}
              \right).
\end{eqnarray*}

Obviously, $\tilde{U}$ is a bilinear map. Further, let $b = \left(
                           \begin{array}{c}
                             1 \\ 1 \\ 0 \\ 1 \\
                           \end{array}
                         \right)$,  $A = \Id$ in the definition of $F$ from Eq. \eqref{equ:defF}, and $e_1$ and $e_2$ the canonical vectors in $\R^2$, i.e.
\begin{equation*}
  e_1 = \left(
            \begin{array}{c}
              1 \\
              0 \\
            \end{array}
          \right),
   \ \ \ \ \
  e_2 = \left(
            \begin{array}{c}
              0 \\
              1 \\
            \end{array}
          \right).
\end{equation*}
Then the following holds
\begin{itemize}
\item[a)] $\tilde{U}(e_1, e_1) = \tilde{U}(e_2, e_1)$,
\item[b)] $F^\prime(e_1, e_1) = 0$,
\item[c)] $F^\prime(e_2, e_1) \neq 0$.
\end{itemize}
Elementary calculations result in
\begin{equation*}
    \tilde{U}(e_1, e_1) = \tilde{U}(e_2, e_1) = \left(
                                   \begin{array}{c}
                                     1  \\ 1 \\ 0 \\ 0 \\
                                   \end{array}
                                 \right).
\end{equation*}
The definition of $F$ from Eq. \eqref{equ:defF} gives
\begin{equation*}
F(x, y) = \frac{1}{3} (\frac{1}{2}(2 (y_1(x_1+x_2))^2 + (x_1y_2)^2 + (x_2y_2)^2) - 2(x_1+x_2)y_1 - x_2y_2).
\end{equation*}
Then
\begin{eqnarray*}
    F(x, e_1) &=& \frac{1}{3} ((x_1 + x_2)^2 - 2(x_1+x_2)), \,
    F(e_1, y) = \frac{1}{3} (y_1^2 + \frac{1}{2}y_2^2 - 2y_1 ) \\
    F(e_2, y) &=& \frac{1}{3} (y_1^2 + \frac{1}{2}y_2^2 - 2y_1 - y_2).
\end{eqnarray*}
and
\begin{equation*}
    F^\prime_x(x, e_1) = \left(
                           \begin{array}{c}
                             \frac{2(x_1 + x_2 - 1)}{3} \\
                             \frac{2(x_1 + x_2 - 1)}{3} \\
                           \end{array}
                         \right), \ \ \ \ \
    F^\prime_y(e_1, y) = \left(
                           \begin{array}{c}
                             \frac{2(y_1-1)}{3} \\
                             \frac{y_2}{3} \\
                           \end{array}
                         \right), \ \ \ \ \
    F^\prime_y(e_2, y) = \left(
                           \begin{array}{c}
                             \frac{2(y_1-1)}{3} \\
                             \frac{y_2-1}{3} \\
                           \end{array}
                         \right).
\end{equation*}
One verifies that $F^\prime(e_1, e_1) = \left(
                              \begin{array}{c}
                                0 \\ 0 \\ 0 \\ 0 \\
                              \end{array}
                            \right)$ and $F^\prime(e_2, e_1) = \left(
                              \begin{array}{c}
                                0 \\ 0 \\ 0 \\ -\frac{1}{3} \\
                              \end{array}
                            \right)$.
\end{proof}

For a convenient understanding, let us briefly repeat the notations
from the ALS method, see Algorithm \ref{alg:ALS}. Let $\mu \in \N_L
\cup \{0\}$, $k \in \N$, and

\begin{eqnarray}\label{eq:defALSseq}
  \ve{p}_{k, \mu} &=&  (p_1^{k+1}, \dots, p_{\mu-1}^{k+1}, p_\mu^{k}, \dots p_L^{k}) \in
  P,\\ \nonumber
  v_{k, \mu} &=& U(\ve{p}_{k, \mu})= U(p_1^{k+1}, \dots, p_{\mu-1}^{k+1}, p_\mu^{k}, \dots
  p_L^{k}) \in \Vp
\end{eqnarray}
be the elements of the sequences $(\ve{p}_{k, \mu})_{k\in\N}$ and
$(v_{k, \mu})_{k\in\N}$ from the ALS algorithm. Note that
$\ve{p}_{k}= \ve{p}_{k,0}=(p_1^k, \dots, p_L^k)$ and
$v_k=U(\ve{p}_k)$.
\begin{defn}[$\Ap(v_k)$]
The set of accumulation points of $(v_k)_{k \in \N}$ is denoted by
$\Ap(v_k)$, i.e.
\begin{equation}\label{eq:AccumulationPoint}
    \Ap(v_k):=\left\{ v \in \Vp : v \mbox{ is an accumulation point of } (v_k)_{k \in \N} \right\}.
\end{equation}
\end{defn}
We demonstrate in Theorem \ref{theo:AcuCrit} that every accumulation
point of $(v_{k})_{k\in\N}$ is a critical point, i.e. $\Ap(v_k)
\subseteq \mathfrak{M}$. This is an existence statement on the
parameter space $P$. Lemma \ref{lemma:pMin} shows us a candidate for
such a parameter system.
\begin{remark}\label{rem:APnotEmpty}
Obviously, if the sequence of parameter $(\ve{p}_k)_{k \in \N}$ is
bounded, then the set of accumulation points of $(\ve{p}_k)_{k \in
\N}$ is not empty. Consequently, the set $\Ap(v_k)$ is not empty,
since the tensor format $U$ is a continuous map.
\end{remark}
\begin{lemma}\label{lemma:pMin}
Let the sequence $(\ve{p}_k)_{k \in \N}$ from the ALS method be bounded and
define for $J \subseteq \N$ the following set of
accumulation points:
\begin{equation*}
    \Ap_{J}:=\bigcup_{\mu=0}^{L-1} \left\{\ve{p} \in P : \ve{p} \mbox{ is an accumulation point of }  (\ve{p}_{k, \mu})_{k \in
    J}\right\}.
\end{equation*}
There exists $\ve{p}^*=(p^*_1, \dots, p^*_L) \in \Ap_{J}$ such that
\begin{equation*}
    \|\ve{p}^*\|= \min_{\ve{p} \in \Ap_{J}}\|\ve{p}\|.
\end{equation*}
\end{lemma}
\begin{proof}
Since the sequence $(\ve{p}_k)_{k \in \N}$ is bounded, it follows
from the definition in Eq. (\ref{eq:defALSseq}) that
$(\ve{p}_{k,\mu})_{k \in \N}$ is also bounded. Therefore the set
$\left\{\ve{p} \in P : \ve{p} \mbox{ is an accumulation point of }
(\ve{p}_{k, \mu})_{k \in J}\right\}$ is not empty and compact. Hence
$\Ap$ is a compact and non-empty set.
\end{proof}
We are now ready to establish our main assumptions on the sequence
from the ALS method.
\begin{assumption} During the article, we say that $(\ve{p}_k)_{k \in
\N}$ satisfies assumption A1 or assumption A2 if the following holds
true:
\begin{itemize}
  \item [{\bf (A1:)}] The sequence $(\ve{p}_k)_{k \in \N}$ is
  bounded.
  \item [{\bf (A2:)}] The sequence $(\ve{p}_k)_{k \in \N}$ is
  bounded and for $J \subseteq \N$ we have
    \begin{equation}\label{eq:rankCond}
        \forall \mu \in \N_L: \exists k_0 \in J : \forall k \in J : \quad k\geq k_0 \quad \Rightarrow \quad  \Rang(W_{k,\mu}) = \Rang(W^*_{\mu}),
    \end{equation}
where $\ve{p}^*=(p^*_1, \dots, p^*_L) \in \Ap_{J}$ is a accumulation
point form Lemma \ref{lemma:pMin} and
\begin{eqnarray*}
  W_{k,\mu} &=& W_\mu(p_1^{k+1},\dots p_{\mu-1}^{k+1}, p_{\mu+1}^{k}, \dots, p_{L}^{k} ), \\
  W^*_{\mu} &=& W_\mu(p_1^{*},\dots p_{\mu-1}^{*}, p_{\mu+1}^{*}, \dots, p_{L}^{*} ).
\end{eqnarray*}
\end{itemize}
\end{assumption}
\begin{remark}
In the proof of Theorem \ref{theo:AcuCrit}, assumption A2 ensures that the ALS method depends continuously
on the parameter system $\ve{p}_{k, \mu}$, i.e. $G_{k, \mu}^+W^T_{k, \mu}b \xrightarrow[k \rightarrow
\infty]{} G^+_\mu W_\mu^Tb$ for a convergent subsequence $(\ve{p}_{k, \mu})_{k \in J}$.
\end{remark}

Using the notations and definitions from Section \ref{sec:ALS}, we
define further
\begin{eqnarray}\label{equ:defAkmu}
  A_{k, \mu}&:=&V_{k, \mu}^T A V_{k, \mu},
\end{eqnarray}
for $k \in \N$ and $\mu \in \N_L$.

For the ALS method there is an explicit formula for the decay of the
values between $f(v_{k, \, \mu+1})$ and $f(v_{k, \, \mu})$. The relation between the function values from
Eq. (\ref{eq:fmicro}) is crucial for the convergence analysis of the ALS method.
\begin{lemma}\label{lemma:fmicro}
Let $k\in \N$, $\mu \in \N_L$. We have
\begin{equation}\label{eq:fmicro}
    f(v_{k, \, \mu+1}) - f(v_{k, \mu }) = -\frac{1}{2} \frac{\dotprod{V_{k, \mu} A_{k, \mu}^{-1} V_{k, \mu}^T r_{k,
    \mu}}{r_{k, \mu }}}{\|b\|^2}
\end{equation}
\end{lemma}
\begin{proof}
From Algorithm \ref{alg:ALS}, we have that $v_{k, \mu+1} = v_{k,
\mu} + \Delta_{k, \mu }$, where $\Delta_{k, \mu }:=V_{k, \mu}A_{k,
\mu}^{-1} V_{k, \mu}^T r_{k, \mu}$. Elementary calculations give

\begin{eqnarray*}
  f(v_{k, \mu+1}) &=& \frac{1}{\|b\|^2}\left[\dotprod{A(v_{k,
\mu} + \Delta_{k, \mu })}{v_{k, \mu} + \Delta_{k, \mu }} -
\dotprod{b}{v_{k, \mu} + \Delta_{k, \mu }}\right]\\
&=& f(v_{k, \mu}) + \frac{- \dotprod{r_{k, \, \mu}}{\Delta_{k, \, \mu}}+
\frac{1}{2}\dotprod{A\Delta_{k, \, \mu}}{\Delta_{k, \, \mu}}}{\|b\|^2}\\
&=&f(v_{k, \mu}) + \frac{- \dotprod{V_{k, \mu} A_{k, \mu}^{-1} V_{k,
\mu}^Tr_{k, \, \mu}}{r_{k, \, \mu}} + \frac{1}{2}\dotprod{AV_{k,
\mu}A_{k, \mu}^{-1} V_{k, \mu}^T r_{k, \mu}}{V_{k, \mu}A_{k,
\mu}^{-1} V_{k, \mu}^T r_{k, \mu}}}{\|b\|^2}\\
&=&f(v_{k, \mu}) + \frac{- \dotprod{V_{k, \mu} A_{k, \mu}^{-1} V_{k,
\mu}^Tr_{k, \, \mu}}{r_{k, \, \mu}} + \frac{1}{2}\dotprod{V_{k, \mu}
A_{k, \mu}^{-1} V_{k, \mu}^Tr_{k, \, \mu}}{r_{k, \, \mu}}}{\|b\|^2}\\
&=&f(v_{k, \mu}) - \frac{1}{2\|b\|^2}\dotprod{V_{k, \mu} A_{k,
\mu}^{-1} V_{k, \mu}^Tr_{k, \, \mu}}{r_{k, \, \mu}}.
\end{eqnarray*}
\end{proof}

\begin{corollary}\label{cor:functionValues}
$(f(v_k))_{k\in \N} \subset \R$ is a descending sequence and there exists $\alpha \in \R$ such that $f(v_k)  \xrightarrow[k
\rightarrow \infty]{} \alpha$.
\end{corollary}
\begin{proof}
Let $k \in \N$ and $\mu \in \N_L$. From Lemma \ref{lemma:fmicro} it
follows that
\begin{eqnarray*}
  f(v_{k+1}) - f(v_k) &=& f(v_{k, L}) - f(v_{k, 0}) =
  \sum_{\mu=1}^L f(v_{k, \, \mu}) - f(v_{k, \mu -1})\\ &=&- \frac{1}{2 \|b\|^2}
  \sum_{\mu=0}^{L-1} \dotprod{ A_{k, \mu}^{-1} V_{k, \mu}^T r_{k,
    \mu}}{V_{k, \mu}^T r_{k, \mu }} \leq 0,
\end{eqnarray*}
since the matrices $A_{k, \mu}^{-1}$ are positive definite. This
shows that $(f(v_k))_{k\in \N} \subset \R$ is a descending sequence.
The sequence of function values $(f(v_k))_{k\in \N}$ is bounded from
below, since the matrix $A$ in the definition of $f$ is positive
definite. Therefore, there exists an $\alpha \in \R$ such that
$f(v_k) \xrightarrow[k \rightarrow \infty]{} \alpha$.
\end{proof}
\begin{lemma}\label{lemma:functionValues}
Let $(v_{k,\mu})_{k\in \N, \mu\in\N_L} \subset \Vp$ be the sequence
from Algorithm \ref{alg:ALS}. We have
\begin{equation}\label{eq:lemmafunctionValues}
f(v_{k,\mu}) = -\frac{1}{2\|b\|^2}\dotprod{
v_{k,\mu}}{b}=-\frac{1}{2\|b\|^2}\|v_{k,\mu}\|^2_A
\end{equation}
for all $k\in \N, \mu\in\N_L$, where
$\dotprod{v}{w}_A:=\dotprod{Av}{w}$ and
$\|v\|_A:=\sqrt{\dotprod{v}{v}_A}$.
\end{lemma}
\begin{proof}
Let $k\in \N$ and $\mu\in\N_L$. We have
\begin{eqnarray*}
  \dotprod{v_{k, \mu}}{v_{k, \mu}}_A &=& \dotprod{A V_{k, \mu-1} \left(V_{k, \mu-1}^T AV_{k, \mu-1}\right)^{-1}V_{k, \mu-1}^T b}{V_{k, \mu-1} \left(V_{k, \mu-1}^T AV_{k, \mu-1}\right)^{-1}V_{k, \mu-1}^T
  b}\\ &=& \dotprod{V_{k, \mu-1}^T b}{\left(V_{k, \mu-1}^T AV_{k, \mu-1}\right)^{-1}V_{k, \mu-1}^T
  b}=\dotprod{v_{k, \mu}}{b}.
\end{eqnarray*}
The rest follows from the definition of $f$, see Eq.
(\ref{equ:deff}).
\end{proof}
\begin{corollary}\label{cor:equf} Let $(v_{k,\mu})_{k\in \N, \mu\in\N_L} \subset \Vp$ be the sequence of represented
tensors from the ALS algorithm. The following holds:
\begin{itemize}
  \item [(a)] $f(v_{k, \mu+1}) \leq f(v_{k, \mu})$,
  \item [(b)] $\|v_{k, \mu+1}\|^2_A \geq \|v_{k, \mu}\|^2_A$,
  \item [(c)] $\dotprod{v_{k,\mu+1}}{b} \geq \dotprod{
v_{k,\mu}}{b}$.
\end{itemize}
\end{corollary}
\begin{proof}
Follows from Lemma \ref{cor:functionValues} and Lemma \ref{lemma:functionValues}.
\end{proof}
\begin{lemma}\label{lemma:gradNull}
Let the sequence $(\ve{p}_k)_{k\in \N} \subset P$ fulfil the
assumption A1. Then we have
\begin{equation*}
    \max_{0\leq \mu\leq L-1}\left\|F_\mu'(p_\mu^k)\right\| \xrightarrow[k \rightarrow \infty]{}0.
\end{equation*}
\end{lemma}
\begin{proof}
According to Lemma \ref{lemma:p_muNew} and Lemma \ref{lemma:fmicro},
we have

\begin{eqnarray*}
  f(v_k)- f(v_{k+1}) &=&\sum_{\mu=1}^L f(v_{k, \mu -1}) - f(v_{k, \, \mu}) = \frac{1}{2 \|b\|^2}
  \sum_{\mu=0}^{L-1} \dotprod{A_{k, \mu}^{-1} V_{k, \mu}^T r_{k,
    \mu}}{V_{k, \mu}^T r_{k, \mu }}\\
    &=&\frac{1}{2 \|b\|^2}
  \sum_{\mu=0}^{L-1} \dotprod{ A_{k, \mu}^{-1} \left(\tilde{U}_{k,
\mu} \tilde{D}^{-\frac{1}{2}}_{k \mu}\right)^TW_{k, \mu}^T r_{k,
    \mu}}{\left(\tilde{U}_{k,
\mu} \tilde{D}^{-\frac{1}{2}}_{k \mu}\right)^TW^T_{k, \mu}r_{k, \mu }}\\
    &=&\frac{1}{2 \|b\|^2}
  \sum_{\mu=0}^{L-1} \dotprod{A_{k, \mu}^{-1} \left(\tilde{U}_{k,
\mu} \tilde{D}^{-\frac{1}{2}}_{k \mu}\right)^T
F'_{\mu}(p_\mu^k)}{\left(\tilde{U}_{k,
\mu} \tilde{D}^{-\frac{1}{2}}_{k \mu}\right)^TF'_{\mu}(p_\mu^k)}\\
&\geq&\frac{1}{2 \|b\|^2}\sum_{\mu=0}^{L-1}
\lambda_{\min}(A^{-1}_{k, \mu}) \lambda_{\min}\left(\tilde{U}_{k,
\mu} \tilde{D}^{-1}_{k \mu} \tilde{U}_{k,
\mu}^T\right)\left\|F_\mu'(p_\mu^k)\right\|^2\\
&\geq&\frac{1}{2 \lambda_{\max}(A)\|b\|^2}\sum_{\mu=0}^{L-1}
\lambda_{\min}\left(\tilde{U}_{k, \mu} \tilde{D}^{-1}_{k \mu}
\tilde{U}_{k, \mu}^T\right)\left\|F_\mu'(p_\mu^k)\right\|^2 ,
\end{eqnarray*}
where $V_{k, \mu}=W_{k, \mu } \tilde{U}_{k, \mu}
\tilde{D}^{-\frac{1}{2}}_{k \mu}$ is from  Eq. (\ref{equ:defVmu}).
In the last estimate, we have used that the Ritz values are bounded
by the smallest and largest eigenvalue of $A$, i.e $\lambda_{\min}
(A) \leq \lambda_{\min} (A_{k, \mu}) \leq \lambda_{\max} (A_{k,
\mu}) \leq \lambda_{\max}(A)$. Since the tensor format $U$ is
continues and the sequence $(\ve{p}_k)_{k \in \N}$ is bounded, it
follows from the theorem of Gershgorin and Cauchy-Schwarz inequality
that there is $\gamma
>0$ such that $\lambda_{\max}\left(W_{k, \mu }^T W_{k, \mu }\right)
\leq \gamma$, recall that $W_{k, \mu }^T W_{k, \mu
}=\tilde{U}_{k,\mu} D_{k,\mu} \tilde{U}^T_{k,\mu}$, see Proposition
\ref{pro:OrthbasisandRepWmu}. Therefore, we have
\begin{equation*}
    f(v_k)- f(v_{k+1}) \geq \frac{1}{2
\lambda_{\max}(A)\gamma
\|b\|^2}\sum_{\mu=0}^{L-1}\left\|F_\mu'(p_\mu^k)\right\|^2 \geq
\frac{1}{2 \lambda_{\max}(A)\gamma \|b\|^2} \max_{0\leq \mu\leq
L-1}\left\|F_\mu'(p_\mu^k)\right\|^2\geq 0.
\end{equation*}

Further, it follows from Corollary \ref{cor:functionValues} that
\begin{equation*}
   0=\lim_{k\rightarrow \infty}  \sqrt{f(v_k)- f(v_{k+1})} = \lim_{k\rightarrow
   \infty} \max_{0\leq \mu\leq L-1}\left\|F_\mu'(p_\mu^k)\right\|.
\end{equation*}
\end{proof}
\begin{theorem}\label{theo:AcuCrit}
Let $(v_k)_{k \in \N}$ be the sequence of represented tensors and
suppose that the sequence of parameter $(\ve{p}_k)_{k \in \N}\subset
P$ from the ALS method fulfils assumption A2. Every accumulation
point of $(v_k)_{k \in \N}$ is a critical point, i.e. $\Ap(v_k)
\subseteq \mathfrak{M}$. Further, we have
\begin{equation*}
    \dist{v_k}  {\mathfrak{M}} \xrightarrow[k \rightarrow \infty]{}0.
\end{equation*}

\end{theorem}
\begin{proof}
Let $\bar{v} \in \Ap(v_k)$ be an accumulation point and $(v_k)_{k\in
J\subseteq\N } \subset U(P)$ a subsequence in the range set of the
tensor format $U$ with $v_k\xrightarrow[k \rightarrow
\infty]{}\bar{v}$. Then there exists $(p_k)_{k\in J\subseteq\N }
\subset P$ with $v_k=U(p_k)$ for all $k \in J$. Further, let $\mu
\in \N_L$ and define for all $k \in J$
\begin{equation*}
    \ve{g}_{k, \mu} := (p_1^{k+1}, \dots, p_{\mu}^{k+1}, p_{\mu+1}^{k}, \dots,
    p_L^k) \in P.
\end{equation*}
Let $\mu^* \in \N_L$ and $(g_{k, \mu^*})_{k \in J' \subseteq J}$
with $g_{k, \mu^*}\xrightarrow[k \rightarrow \infty]{}\ve{p}^*:=\argmax_{\ve{p}\in \Ap_J}\|\ve{p}\| \in P$,  see
Lemma \ref{lemma:pMin}. Without loss of generality, let us assume
that $\mu^*=1$. This assumption makes the the notations not more
complicated then necessary. Since $(g_{k, \mu })_{k \in J}$ is
bounded, there exists $\ve{p}^{[\mu]} \in P$ and a corresponding
subsequence $(\ve{g}_{k, \mu})_{k \in J_{\mu} \subseteq J}$ such
that
\begin{eqnarray*}
  \ve{p}_k= \ve{g}_{k,0}&=&(p_1^k, p_2^k, \dots, p_L^k)\xrightarrow[k \rightarrow
  \infty]{}\ve{p}^*=(p_1, p_2, \dots, p_L),\\
  \ve{g}_{k,1}&=&(p_1^{k+1},p_2^k, \dots, p_L^k)\xrightarrow[k \rightarrow
  \infty]{}\ve{p}^{[1]}=(\tilde{p}_1, p_2, \dots, p_L),\\
  \ve{g}_{k,\mu}&=&(p_1^{k+1}, \dots, p_{\mu}^{k+1}, p_{\mu+1}^{k}, \dots,
    p_L^k)\xrightarrow[k \rightarrow
  \infty]{}\ve{p}^{[\mu]}=(\tilde{p}_1, \dots,\tilde{p}_{\mu}, p_{\mu+1},\dots, p_L),\\
  \ve{g}_{k,L}&=&(p_1^{k+1}, \dots, p_{L}^{k+1})\xrightarrow[k \rightarrow
  \infty]{}\ve{p}^{[L]}=(\tilde{p}_1, \dots,\tilde{p}_{L}).
\end{eqnarray*}
From Lemma \ref{lemma:gemLimitPoint} and $U(\ve{p}_k)\xrightarrow[k \rightarrow
\infty]{}\bar{v}$ it follows that
\begin{eqnarray}\label{equ:Gl1Proof}
          \bar{v}&=&\lim_{k \rightarrow \infty}U(\ve{g}_{k, \mu}) = U(\ve{p}^{[\mu]})\quad \mbox{f.a. }\mu \in \N_L,
\end{eqnarray}
where  $k \in J$.
Furthermore, we have
\begin{equation*}
    \ve{p}^*=\ve{p}^{[1]}= \dots = \ve{p}^{[L]}.
\end{equation*}
To show this, assume that
\begin{equation*}
    M:=\left\{ \mu \in \N_L : \ve{p}^{[\mu]}\neq
    \ve{p}^*\right\}\neq \emptyset
\end{equation*}
and define $\nu:=\min M \in \N_L$. From assumption A2 it follows
\begin{equation*}
    p_{\nu}^{k+1} = G_{k, \nu}^+W^T_{k, \nu}b \xrightarrow[k \rightarrow
\infty]{} G^+_\nu W_\nu^Tb=\tilde{p}_\nu.
\end{equation*}
Thus we have in particular that $\tilde{p}_\nu \bot \Kern W_\nu$, see the Definition of $G^+_\nu$ and Proposition \ref{pro:OrthbasisandRepWmu}.
Since $U(\ve{p}^*)=U(\ve{p}^{[\nu]}) \Leftrightarrow W_\nu p_\nu = W_\nu \tilde{p}_\nu$, it follows further that $\delta_\nu := p_\nu - \tilde{p}_\nu \in \Kern W_\nu$ and $\|p_\nu\|^2 = \|\tilde{p}_\nu\|^2 + \|\delta_\nu\|^2$.
Lemma \ref{lemma:p_muNew} and Lemma \ref{lemma:gradNull} show that
\begin{equation*}
    W_{k, \nu}^T(A W_{k, \nu} p_\nu^k - b) \xrightarrow[k \rightarrow
\infty]{} 0.
\end{equation*}
From the definition of $\nu$, we have then
\begin{equation*}
    W_\nu^T(A W_\nu p_\nu -b)=0,
\end{equation*}
note that for $\nu=\min M$
\begin{equation*}
    \ve{g}_{k,\nu}=(p_1^{k+1}, \dots, p_{\nu}^{k+1}, p_{\nu+1}^{k}, \dots,
    p_L^k)\xrightarrow[k \rightarrow
  \infty]{}\ve{p}^{[\nu]}=(p_1, \dots, p_{\nu-1}, \tilde{p}_\nu ,p_{\nu+1},\dots, p_L)
\end{equation*}
holds. Since $\tilde{p}_\nu = G^+_\nu W_\nu^Tb$  and $\ve{p}^* = \argmin_{\ve{p} \in \Ap_J}\|\ve{p}\|$, it follows $\|\tilde{p}_\nu\| = \|p_\nu\|$.
Hence, we have $\tilde{p}_\nu=p_\nu$, because $\|p_\nu\|^2 = \|\tilde{p}_\nu\|^2 + \|\delta_\nu\|^2$ implies then $\delta_\nu=0$. But $\tilde{p}_\nu=p_\nu$ contradicts the definition of $\nu=\min M$. Consequently, we have
\begin{equation*}
    \ve{p}^*=\ve{p}^{[1]}= \dots = \ve{p}^{[L]}.
\end{equation*}
From Eq. (\ref{equ:Gl1Proof}) and the definition of $p_\mu^{k+1}$ it follows then
\begin{equation*}
    \bar{v}=U(\ve{p}^*)
\end{equation*}
and
\begin{equation*}
    0= \lim_{k \rightarrow \infty}F'_\mu(\ve{g}_{k, \mu}) = F'_\mu(\ve{p}^*) \quad \mbox{ for all } \mu \in \N_L,
\end{equation*}
i.e. $\Ap(v_k) \subseteq \mathfrak{M}$. Now, let $\delta_k=\inf_{v \in \mathfrak{M}}\|v_k-v\|$ and suppose that there exists a subsequence $(\delta_k)_{k \in J \subseteq \N}$ with
$\lim_{k \rightarrow \infty} \delta_k = \delta \in \R_+$. Then $(v_k)_{k \in J}$ has a convergent subsequence. Since this subsequence must have its limit point in $\mathfrak{M}$,
it follows that $\delta=0$. Which proves $\dist{v_k}  {\mathfrak{M}} \xrightarrow[k \rightarrow \infty]{}0$ by contradiction.

\end{proof}
\begin{lemma}\label{lemma:dist}
Let $(v_k)_{k\in \N} \subset \Vp$ be the sequence of represented
tensors from the ALS method. It holds
\begin{equation*}
    \|v_{k+1} - v_{k}\|_A\xrightarrow[k \rightarrow \infty]{}0.
\end{equation*}
\end{lemma}

\begin{proof}
 Let $k \in \N$. We have
\begin{eqnarray}\label{eq:profvkNull}
  \|v_{k+1} - v_{k}\|_A^2 &=& \left\| \sum_{\mu=1}^L v_{k,\mu} - v_{k,
  \mu-1}\right\|^2_A \leq \left( \sum_{\mu=1}^L \|v_{k,\mu} - v_{k,
  \mu-1}\|_A \right)^2 \leq L \sum_{\mu=0}^{L-1} \left\|v_{k,\mu+1} - v_{k,
  \mu}\right\|^2_A.
\end{eqnarray}
Since $v_{k, \mu+1} - v_{k, \mu} =
V_{k,\mu}A_{k,\mu}^{-1}V_{k,\mu}^T r_{k,\mu}$, it follows further
\begin{eqnarray*}
   \left\|v_{k,\mu+1} - v_{k,
  \mu}\right\|^2_A &=& \left\| V_{k,\mu}A_{k,\mu}^{-1}V_{k,\mu}^T r_{k,\mu} \right\|^2_A = \dotprod{A V_{k,\mu}A_{k,\mu}^{-1}V_{k,\mu}^T r_{k,\mu}}{V_{k,\mu}A_{k,\mu}^{-1}V_{k,\mu}^T r_{k,\mu}}\\
  &=& \dotprod{A_{k,\mu}^{-1}V_{k,\mu}^T r_{k,\mu}}{V^T_{k,\mu}AV_{k,\mu}A_{k,\mu}^{-1}V_{k,\mu}^T r_{k,\mu}}=\dotprod{A_{k,\mu}^{-1}V_{k,\mu}^T r_{k,\mu}}{V_{k,\mu}^T r_{k,\mu}}\\
  &=&\dotprod{V_{k,\mu}A_{k,\mu}^{-1}V_{k,\mu}^T r_{k,\mu}}{r_{k,\mu}}.
\end{eqnarray*}
Combining this with Eq. (\ref{eq:fmicro}) and (\ref{eq:profvkNull})
gives
\begin{eqnarray*}
  \|v_{k+1} - v_{k}\|_A^2 &\leq& 2 L \|b\|^2 \sum_{\mu=0}^{L-1} \left(f(v_{k, \mu}) - f(v_{k, \mu+1})\right) = 2 L \|b\|^2 \left(f(v_{k})- f(v_{k+1})\right).
\end{eqnarray*}
From Corollary \ref{cor:functionValues} it follows $\left(f(v_{k}) -
f(v_{k})\right)\xrightarrow[k \rightarrow \infty]{} 0$. Therefore,
we have

\begin{equation*}
    \|v_{k+1} - v_{k}\|_A\xrightarrow[k \rightarrow \infty]{}0.
\end{equation*}
\end{proof}
\begin{lemma}\label{lemma:gemLimitPoint}
Let $(v_k)_{k\in \N} \subset \Vp$ be the sequence of tensors from
the ALS method and $\bar{v} \in\Vp$ with $\lim_{k \rightarrow
\infty} v_k = \bar{v}$.
Further, let $\mu \in \N_L$ and $(v_{k, \mu})_{k\in \N} \subset \Vp$
as defined in Algorithm \ref{alg:ALS}. We have
\begin{equation*}
    \lim_{k \rightarrow \infty} v_{k, \mu} = \bar{v} \quad
    \mbox{for all } \mu \in \N_L.
\end{equation*}
\end{lemma}
\begin{proof}
Define $v_{k, 0}:=v_k$ (like in Algorithmus \ref{alg:ALS}) and
assume that

\begin{equation*}
    M:=\left\{ \mu \in \N_L : \lim_{k \rightarrow \infty} v_{k, \mu} \neq
    \bar{v}\right\}\neq\emptyset.
\end{equation*}
Furthermore, set $\mu^*:=\min M \in \N_L$.

From Lemma \ref{lemma:dist} it follows
\begin{equation*}
    \|v_{k,\mu^*} - \bar{v}\|_A \leq \|v_{k,\mu^*} - v_{\mu^*-1, k}\|_A + \|v_{k, \mu^*-1} - \bar{v}\|_A\xrightarrow[k \rightarrow
    \infty]{}0.
\end{equation*}
But this contradicts the definition of $\mu^*$.
\end{proof}

In the following, the dimension of the tensor space
$\Vp=\Ten_{\mu=1}^d \R^{m_\mu}$ is denoted by $N \in \N$, i.e.
$N=\prod_{\mu=1}^d m_\mu$. The statement of Lemma
\ref{lemma:recursionTan} delivers an explicit recursion formula for
the tangent of the angle between iteration points and an arbitrary
tensor. This result is important for the rate of convergence of the
ALS algorithm.
\begin{lemma}\label{lemma:ALSrecursion}
Let $\mu, \nu \in \N_L$, $\nu\neq \mu$, and $\ve{p}=(p_1, \dots,
p_\nu, \dots, p_\mu, \dots, p_L) \in P$. There exists a multilinear
map $M_{\mu, \nu} : P_1 \times \dots \times P_{\nu-1} \times
P_{\nu+1} \times \dots  \times P_{\mu-1} \times P_{\mu+1} \times
\dots \times P_L \times \Vp \rightarrow L(P_{\nu}, P_\mu)$ such that
\begin{equation}\label{eq:recursion}
    W^T_{\mu}(p_1, \dots, \mathbf{g_\nu}, \dots, p_{\mu-1}, p_{\mu+1}, \dots, p_L)b =M_{\mu,
    \nu}(p_1,
\dots, p_{\nu-1}, p_{\nu+1}, \dots, p_{\mu-1}, p_{\mu+1}, \dots,
p_L, b) \mathbf{g_\nu}
\end{equation}
for all $\mathbf{g_\nu} \in P_\nu$. Moreover, we have
\begin{equation*}
    M_{\mu,
    \nu}(p_1,
\dots, p_{\nu-1}, p_{\nu+1}, \dots, p_{\mu-1}, p_{\mu+1}, \dots,
p_L, b) = M^T_{\nu,
    \mu}(p_1,
\dots, p_{\nu-1}, p_{\nu+1}, \dots, p_{\mu-1}, p_{\mu+1}, \dots,
p_L, b).
\end{equation*}
\end{lemma}
\begin{proof}
Follows form Proposition \ref{pro:OrthbasisandRepWmu} (v) and
definition of $W^T_{\mu}(p_1, \dots, \mathbf{g_\nu}, \dots,
p_{\mu-1}, p_{\mu+1}, \dots, p_L)b$.
\end{proof}

\begin{corollary}\label{cor:ALSrecursion}
Let $\mu \in \N_L$, $k \geq 2$, and $\ve{p}_{k, \mu}= (p^{k+1}_1,
\dots, p^{k+1}_{\mu-1}, p_\mu^k,p^{k}_{\mu+1}, \dots, p_L^k)$ form
Algorithm \ref{alg:ALS}. There exists a multilinear map $M_{\mu} :
P_1 \times \dots \times P_{\mu-2} \times P_{\mu+1} \times \dots
\times P_L \times \Vp \rightarrow L(P_{\mu-1}, P_\mu)$ such that

\begin{eqnarray}\label{eq:recursion1}
    p^{k+1}_\mu&=&G^+_{k, \, \mu} M_\mu(p_1^{k+1},
\dots, p_{\mu-2}^{k+1}, p_{\mu+1}^k, \dots, p_L^k, b)
p_{\mu-1}^{k+1},\\
    p_{\mu-1}^{k+1}&=&G^+_{k, \, \mu-1} M^T_\mu(p_1^{k+1},
\dots, p_{\mu-2}^{k+1}, p_{\mu+1}^k, \dots, p_L^k, b) p_{\mu}^{k},\\
\nonumber \mbox{ i.e. } p^{k+1}_\mu&=&G^+_{k, \, \mu}
M_\mu(p_1^{k+1}, \dots, p_{\mu-2}^{k+1}, p_{\mu+1}^k, \dots, p_L^k,
b) \, G^+_{k, \, \mu-1} M^T_\mu(p_1^{k+1}, \dots, p_{\mu-2}^{k+1},
p_{\mu+1}^k, \dots, p_L^k, b) p_{\mu}^{k},
\end{eqnarray}
where $G^+_{k, \, \mu-1}$ and $G^+_{k, \, \mu}$ are defined in
Algorithm \ref{alg:ALS}.
\end{corollary}
\begin{proof}
Follows form Eq. (\ref{eq:defPmu}) and Lemma
\ref{lemma:ALSrecursion}.
\end{proof}
The following example shows a concrete realisation of the matrix
$M_{\mu}$ for the tensor rank-one approximation problem.
\begin{example}\label{exa:rank1}
The approximation of $b\in \Vp$ by a rank one tensor is considered.
Let $v_k = p_1^k \ten p_2^k \ten \dots \ten p_d^k$ and
\begin{equation*}
    b= \sum_{i_1=1}^{t_1} \dots \sum_{i_d=1}^{t_d} \beta_{(i_1, \dots, i_d)} \Ten_{\mu=1}^d b_{\mu,
    i_\mu},
\end{equation*}
i.e. the tensor $b$ is given in the Tucker decomposition. From Eq.
(\ref{eq:defPmu}) it follows
\begin{eqnarray*}
  p_1^{k+1} &=& \frac{1}{\prod_{\mu=2}^d \left\|p_\mu^k\right\|^2}\sum_{i_1=1}^{t_1} \dots \sum_{i_d=1}^{t_d} \beta_{(i_1, \dots,
  i_d)} \prod_{\mu=2}^d \dotprod{b_{\mu, i_\mu}}{p_\mu^k} b_{1,
  i_1}\\
  &=& \frac{1}{\prod_{\mu=2}^{d-1} \left\|p_\mu^k\right\| \|p_d^k\|^2} \left[\sum_{i_1=1}^{t_1} \sum_{i_d=1}^{t_d}  b_{1, i_1} \sum_{i_2=1}^{t_2}\dots \sum_{i_{d-1}=1}^{t_{d-1}} \beta_{(i_1, \dots,
  i_d)} \prod_{\mu=2}^{d-1} \frac{\dotprod{b_{\mu, i_\mu}}{p_\mu^k}}{\|p_\mu^k\|} b^T_{d,
  i_d}\right]p_d^k\\
  &=&\frac{1}{\prod_{\mu=2}^{d-1} \left\|p_\mu^k\right\| \|p_d^k\|^2}\underbrace{B_1
  \Gamma_{1,k} B_d^T}_{M_{1}(p_2^k, \dots, p_{d-1}^k)=} p_d^k,
\end{eqnarray*}
where $B_\mu=\left( b_{\mu, i_\mu}\,:\, 1\leq i_\mu \leq
t_\mu\right) \in \R^{n_\mu \times t_\mu}$,
$B_\mu^TB_\mu=\Id_{\R^{t_\mu}}$, and the entries of the matrix
$\Gamma_{1,k} \in \R^{t_1 \times t_d}$ are defined by
\begin{equation*}
    [\Gamma_{1, k}]_{i_1, i_d} = \sum_{i_2=1}^{t_2}
\dots \sum_{i_{d-1}=1}^{t_{d-1}} \beta_{(i_1, \dots,
  i_d)} \prod_{\mu=2}^{d-1} \frac{\dotprod{b_{\mu,
i_\mu}}{p_\mu^k}}{\|p_\mu^k\|} \quad \left( 1\leq i_1 \leq t_1, \, 1
\leq i_d \leq t_d\right).
\end{equation*}
Note that $\Gamma_{1,k}$ is a diagonal matrix if the coefficient
tensor $\beta \in \Ten_{\mu=1}^d \R^{t_\mu}$ is super- diagonal, see
the example in \cite{AramMikeALS1}.
For $p_d^k$ it follows further
\begin{eqnarray*}
  p_d^{k} &=& \frac{1}{\prod_{\mu=1}^{d-1} \left\|p_\mu^k\right\|^2}\sum_{i_1=1}^{t_1} \dots \sum_{i_d=1}^{t_d} \beta_{(i_1, \dots,
  i_d)} \prod_{\mu=1}^{d-1} \dotprod{b_{\mu, i_\mu}}{p_\mu^k} b_{d,
  i_d}
  =\frac{1}{\left\|p_1^k\right\|^2\prod_{\mu=2}^{d-1} \left\|p_\mu^k\right\|}B_d
  \Gamma^T_{1,k} B_1^T p_1^k
\end{eqnarray*}
and finally
\begin{equation*}
    p_1^{k+1}= \frac{1}{\prod_{\mu=1}^d \left\|p_\mu^k\right\|^2}\, B_1
  \Gamma_{1,k}\Gamma^T_{1,k} B_1^T \,p_1^k.
\end{equation*}
\end{example}

\begin{lemma}\label{lemma:RecursionV}
Let $\mu \in \N_L$, $(\ve{p}_{k, \mu})_{k \in \N} \subset P$, and
$\left(v_{k, \mu}\right)_{k \in \N} \subset \Vp$ be the sequences
from Algorithm \ref{alg:ALS}. Furthermore, define
\begin{eqnarray*}
  M_{k, \mu} &:=& M_\mu(p_1^{k+1},
\dots, p_{\mu-2}^{k+1}, p_{\mu+1}^k, \dots, p_L^k, b),\\
  H_{k, \mu-1} &:=& W^T_{k, \mu-1}W_{k, \mu-1},\\
  N_{k, \mu}&:=& W_{k, \mu } G^+_{k, \mu} M_{k, \mu} H^+_{k,\mu-1} W^T_{k, \mu-1},
\end{eqnarray*}
where we have used the notations from Algorithm \ref{alg:ALS}. A
micro-step of the ALS method is described by the following recursion
formula:
\begin{equation}\label{eq:recursionV}
    v_{k,\mu+1} = N_{k, \mu} v_{k,\mu} \quad \mbox{ for all }
    k\geq2,
\end{equation}
i.e.
\begin{equation*}
    U(p_1^{k+1}, \dots, p_{\mu-1}^{k+1}, p_{\mu}^{k+1}, p_{\mu+1}^{k}, \dots, p_{L}^{k}) = N_{k, \mu} \, U(p_1^{k+1}, \dots, p_{\mu-1}^{k+1}, p_{\mu}^{k}, p_{\mu+1}^{k}, \dots, p_{L}^{k}).
\end{equation*}
\end{lemma}
\begin{proof}
According to Corollary \ref{cor:ALSrecursion}, Remark
\ref{remark:pOrthKern}, and definition of $v_{k, \mu+1}$, we have
that
\begin{eqnarray*}
  N_{k, \mu} v_{k, \mu} &=& W_{k, \mu } G^+_{k, \mu} M_{k, \mu} H^+_{k,\mu-1} W^T_{k,
  \mu-1} W_{k, \mu-1} p_{\mu-1}^{k+1}
  = W_{k, \mu } G^+_{k, \mu} M_{k, \mu} H^+_{k,\mu-1}  H_{k, \mu-1}
  p_{\mu-1}^{k+1}\\
  &=& W_{k, \mu } G^+_{k, \mu} M_{k, \mu} p_{\mu-1}^{k+1} = W_{k, \mu } G^+_{k, \mu} W^T_{k, \mu} b = W_{k, \mu } p_{\mu}^{k+1} = v_{k,
  \mu+1}.
\end{eqnarray*}

\end{proof}
\begin{lemma}\label{lemma:recursionTan}
Let $\left(N_{k, \mu}\right)_{k \in \N,\, \mu\in \N_L}$ be the
sequence of matrices from Lemma \ref{lemma:RecursionV}, $\bar{v} \in
\Vp \setminus\{0\}$, and $R \in \R^{N \times N-1}$ an orthogonal
matrix with $\Span(\bar{v})^\bot = \Bild(R)$, i.e. the column
vectors of $R$ form an orthonormal basis of the linear space
$\Span(\bar{v})^\bot$. Assume further that $c_{k,\mu}:=
\frac{\bar{v}^T}{\|\bar{v}\|}v_{k, \mu} \in \R \setminus\{0\}$ and
$s_{k, \mu }:= R^Tv_{k, \mu} \in \R^{N-1} \setminus\{0\}$ holds
true. Then we have the following recursion formula for the tangent
of the angles:
\begin{equation*}
    \left| \tan\angle\left[\bar{v}, v_{k, \mu+1}\right] \right|= \left|\frac{q^{(s)}_{k,\mu}}{q^{(c)}_{k,\mu}}\right|\left| \tan\angle\left[\bar{v}, v_{k, \mu}\right]
    \right|,
\end{equation*}
where
\begin{eqnarray*}
  q_{k,\mu}^{(s)} &:=&\frac{\left\|R^T N_{k, \mu} \ve{v} \,\, c_{k,\mu} +  R^T N_{k, \mu} R \,\, s_{k,\mu} \right\|}{\|s_{k,\mu}\|},  \\
  q_{k,\mu}^{(c)} &:=&\frac{\left| \ve{v}^T N_{k, \mu} \ve{v} \,\, c_{k,\mu} + \ve{v}^T N_{k, \mu} R \,\,
  s_{k,\mu}\right|}{|c_{k,\mu}|}.
\end{eqnarray*}

\end{lemma}

\begin{proof}
The block matrix
\begin{equation*}
    V:=\left[
         \begin{array}{cc}
         \ve{v} & R\\
         \end{array}
       \right]\in \R^{N \times N}, \quad  \left(\,\ve{v}:=\bar{v}/\|\bar{v}\| \,\right).
\end{equation*}
is orthogonal, i.e. the columns of the matrix $V$ build an
orthonormal basis of the tensor space $\Vp$. The tensor $v_{k,\mu}$
and the matrix $N_{k,\mu}$ are represented with respect to the basis
$V$, i.e
\begin{eqnarray*}
  v_{k, \mu} &=& V V^Tv_{k, \mu}= \left[
         \begin{array}{cc}
         \ve{v} & R\\
         \end{array}
       \right]\left(
                \begin{array}{c}
                  \ve{v}^Tv_{k,\mu} \\
                  R^T v_{k,\mu} \\
                \end{array}
              \right)= \left[
         \begin{array}{cc}
         \ve{v}& R\\
         \end{array}
       \right]\left(
                \begin{array}{c}
                  c_{k,\mu} \\
                  s_{k,\mu} \\
                \end{array}
              \right)
\end{eqnarray*}
and
\begin{eqnarray*}
  N_{k, \mu} &=& V \left(V^T N_{k, \mu} V \right)V^T = \left[
         \begin{array}{cc}
         \ve{v} & R\\
         \end{array}
       \right] \left[
       \begin{array}{cc}
       \ve{v}^T N_{k, \mu} \ve{v} &  \ve{v}^T N_{k, \mu} R\\
       R^T N_{k, \mu} \ve{v}&  R^T N_{k, \mu} R\\
       \end{array}\right]
  \left[
         \begin{array}{cc}
         \ve{v} & R\\
         \end{array}
       \right]^T.
\end{eqnarray*}
The recursion formula (\ref{eq:recursionV}) leads to the recursion
of the coefficient vector
\begin{eqnarray*}
\left( \begin{array}{c}
        c_{k+1,\mu} \\
        s_{k+1,\mu} \\
        \end{array}
        \right) = \left[
       \begin{array}{cc}
       \ve{v}^T N_{k, \mu} \ve{v} &  \ve{v}^T N_{k, \mu} R\\
       R^T N_{k, \mu} \ve{v}&  R^T N_{k, \mu} R\\
       \end{array}\right]\left(
                \begin{array}{c}
                  c_{k,\mu} \\
                  s_{k,\mu} \\
                \end{array}
              \right)=\left(
                \begin{array}{c}
                  \ve{v}^T N_{k, \mu} \ve{v} \,\, c_{k,\mu} + \ve{v}^T N_{k, \mu} R \,\, s_{k,\mu}\\
                  R^T N_{k, \mu} \ve{v} \,\, c_{k,\mu} +  R^T N_{k, \mu} R \,\, s_{k,\mu} \\
                \end{array}
              \right).
\end{eqnarray*}
Since $\|s_{k, \mu}\| \neq 0$ and $|c_{k, \mu}| \neq 0$ we have
\begin{eqnarray*}
  \tan^2\angle[\bar{v}, v_{k, \mu+1}]  &=& \frac{\dotprod{RR^T v_{k, \mu+1}}{v_{k, \mu+1}}}{\dotprod{\ve{v}\ve{v}^T v_{k, \mu+1}}{v_{k,
  \mu+1}}} = \frac{\|R^T v_{k, \mu+1}\|^2}{\left(\ve{v}^T v_{k,
  \mu+1}\right)^2}=\frac{\|s_{k, \mu+1}\|^2}{(c_{k, \mu+1})^2}=\frac{\left(q_{k,\mu}^{(s)}\right)^2}{\left(q_{k,\mu}^{(c)}\right)^2}\frac{\|s_{k, \mu}\|^2}{\left(c_{k,
  \mu}\right)^2}\\
  &=&\left(\frac{q_{k,\mu}^{(s)}}{q_{k,\mu}^{(c)}}\right)^2
  \frac{\|R^T v_{k, \mu}\|^2}{\left(\ve{v}^T v_{k,
  \mu}\right)^2} = \left(\frac{q_{k,\mu}^{(s)}}{q_{k,\mu}^{(c)}}\right)^2 \tan^2\angle[\bar{v}, v_{k, \mu}].
\end{eqnarray*}
\end{proof}
\begin{theorem}\label{theorem:isoAP}
Suppose that the sequence $(\ve{p}_k)_{k \in \N}\subset P$ from
Algorithm \ref{alg:ALS} fulfils assumption A1. If one accumulation
point $\bar{v} \in \Ap(v_k)\neq\emptyset$ is isolated, then we have
\begin{equation*}
    v_k \xrightarrow[k \rightarrow \infty]{}\bar{v}.
\end{equation*}
\end{theorem}
Furthermore, we have that either the ALS method converges after
finitely many iteration steps or
\begin{equation*}
    \left|\tan\angle[\bar{v}, v_{k, \mu+1}]\right| \leq q_\mu \left|\tan\angle[\bar{v}, v_{k, \mu}]
    \right|,
\end{equation*}
where
\begin{equation*}
    q_\mu:= \limsup_{k \rightarrow
    \infty}\left| \frac{q_{k,\mu}^{(s)}}{q_{k,\mu}^{(c)}}\right|.
\end{equation*}
\begin{proof}
Let $\varepsilon > 0$ such that $\bar{v}$ is the only accumulation
point in $\bar{U}:=\left\{v\in \Vp :  \|\bar{v} - v\|_A \leq
\varepsilon\right\}$. Assuming that the sequence $(v_k)_{k \in
\N}\subset \Vp$ from the ALS algorithm does not converge to
$\bar{v}$ and let $\Ip \subset \N$ be a subset with
\begin{equation*}
    \|\bar{v}- v_k \|_A\leq \varepsilon
\end{equation*}
for all $k \in \Ip$. Since $\bar{v}$ is the only accumulation in
$\bar{U}$ and $(v_k)_{k \in \N}$ does not converge to $\bar{v}$ the
following set $\Ip_k$ is for all $k \in \Ip$ well-defined and
finite:
\begin{equation*}
    \Ip_k:=\left\{ k'\in \N : \, \|\bar{v} - v_{\tilde{k}}\|_A\leq \varepsilon \mbox{ for all } k\leq \tilde{k} \leq k'\right\}.
\end{equation*}
The definition of the map $k' : \Ip \rightarrow \N, \, k \mapsto
k'(k):=\max \Ip_k$ implies that
\begin{eqnarray*}
  \|\bar{v}- v_{k'(k)}\|_A \leq \varepsilon \, \mbox{ and } \,
  \|\bar{v}- v_{k'(k)+1}\|_A > \varepsilon
\end{eqnarray*}
for all $k \in \Ip$. Since $\bar{v}$ is the only accumulation point
of $(v_k)_{k\in \N}$ in $\bar{U}$ it follows that the subsequence
$(v_{k'(k)})_{k \in \Ip}$ converges to $\bar{v}$. Therefore, we have
\begin{equation*}
    \|\bar{v}- v_{k'(k)}\|_A \leq \frac{\varepsilon}{2}
\end{equation*}
and
\begin{equation*}
    \|v_{k'(k)+1}- v_{k'(k)}\|_A \geq \|\bar{v}- v_{k'(k)+1}\|_A - \|\bar{v}- v_{k'(k)}\|_A \geq \frac{\varepsilon}{2}
\end{equation*}
for sufficient large $k \in \Ip$. But this contradicts the statement
$\|v_{k+1} - v_{k}\|_A\xrightarrow[k \rightarrow \infty]{}0$ from
Lemma \ref{lemma:dist}. The inequality for the rate of convergence
of an ALS micro-step $\left|\tan\angle[\bar{v}, v_{k, \mu+1}]\right|
\leq q_\mu \left|\tan\angle[\bar{v}, v_{k, \mu}]\right|$ follows
direct from Lemma \ref{lemma:recursionTan} and the definition of
$q_\mu$. Note that in Lemma \ref{lemma:recursionTan}, $c_{k,
\mu}\neq0$ since $\lim_{k \rightarrow \infty} v_{k, \mu}=\bar{v}$.
If $s_{k_0, \mu}=0$ for some $k_0\in \N$, then the ALS method
converges after finitely many iteration steps.
\end{proof}
\begin{corollary}
Suppose that the sequence $(\ve{p}_k)_{k \in \N}\subset P$ fulfils
A2 and assume that the set of critical points $\mathfrak{M}$ is
discrete,\footnote{In topology, a set which is made up only of
isolated points is called discrete.} then the sequence of
represented tensors $(v_k)_{k \in \N}$ from the ALS method is
convergent.
\end{corollary}
\begin{proof}
Follows directly from Theorem \ref{theorem:isoAP} and Theorem
\ref{theo:AcuCrit}.
\end{proof}
\begin{remark}\label{rem:convergenceRate}$ $
\begin{itemize}
  \item The convergence rate for an entire ALS iteration step is given by $q:=\prod_{\mu=1}^L q_{\mu-1}$, since
    \begin{equation*}
    \left|\tan\angle[\bar{v}, v_{k+1}]\right| = \left|\tan\angle[\bar{v}, v_{k,L}]\right| \leq q_{L-1} \left|\tan\angle[\bar{v}, v_{k,L-1}]\right| \leq \prod_{\mu=1}^L q_{\mu-1} \left|\tan\angle[\bar{v}, v_{k,0}]\right| = q\left|\tan\angle[\bar{v}, v_{k}]\right|.
    \end{equation*}
\item Without further assumptions on the tensor $b$ from Eq. (\ref{equ:deff}), one cannot say more about the rate of convergence. But the ALS method can converge sublinearly, Q-linearly, and even Q-superlinearly.
We refer the reader to \cite{OR70} for a detailed description of
convergence speed.
    \begin{description}
      \item[-] If $q=0$, then the sequence $\left( \left|\tan\angle[\bar{v}, v_{k}]\right|\right)_{k \in \N}$ converges Q-superlinearly.
      \item[-] If $q <1$, then the sequence $\left( \left|\tan\angle[\bar{v}, v_{k}]\right|\right)_{k \in \N}$ converges at least Q-linearly.
      \item[-] If $q= 1$, then the sequence $\left(\left|\tan\angle[\bar{v}, v_{k}]\right|\right)_{k \in \N}$ converges  sublinearly.
    \end{description}

A specific tensor format $U$ has practically no impact on the
different convergence rates. Since we can find explicit examples for
all cases already for rank-one tensors. Please note that the
representation of rank-one tensors is included in all tensor formats
of practical interest.
      In the following, we give a brief overview of our results about the convergence rates for the tensor rank-one approximation, please see \cite{AramMikeALS1} for proofs and detailed description.
      The multilinear map that describes the representation of rank-one tensors is given by

      \begin{eqnarray*}
        U: \bigtimes_{\mu=1}^d \R^n&\rightarrow& \Ten_{\mu=1}^d \R^n\\
        (p_1, \dots, p_d) &\mapsto& U(p_1, \dots, p_d)=\Ten_{\mu=1}^d p_\mu.
      \end{eqnarray*}

      A tensor $b$ is called totally orthogonal decomposable if
      there exist $r \in \N$ with
      \begin{equation*}
        b = \sum_{j=1}^r \lambda_j \Ten_{\mu=1}^d b_{\mu, j} \quad (b_{\mu,j} \in \R^n)
      \end{equation*}
      such that for all $\mu \in \N_{d}$ and $j_1, j_2 \in \N_r$ the following holds:
      \begin{equation*}
        \dotprod{b_{\mu, j_1}}{b_{\mu, j_1}} = \delta_{j_1, j_2}.
      \end{equation*}
    The set of all totally orthogonal decomposable tensors is denoted by
    \begin{equation*}
        \mathcal{TO}=\left\{b \in \Vp : b \mbox{ is totally orthogonal
        decomposable}\right\} \subset \Vp.
    \end{equation*}
It is shown in \cite{AramMikeALS1} that the tensor rank-one
approximation of every $b \in \mathcal{TO}$ by means of the ALS
method converges Q-superlinearly, i.e. $q=0$.\\

For examples of Q-linear and sublinear convergence, we will consider
the tensor $b_\lambda \in \Vp$  given by
\begin{equation*}
    b_\lambda = \Ten_{\mu=1}^3 p + \lambda \left(p \ten q \ten q + q \ten p \ten q + q \ten q \ten p \right)
\end{equation*}
for some $\lambda \in \R_{\geq 0}$ and $p, q \in \R^n$ with $\|p\| =
\|q\| = 1$, $\dotprod{p}{q} = 0$. If $\lambda \leq \frac{1}{2}$, it
is shown in \cite{AramMikeALS1} that $\bar{v} = \Ten_{\mu=1}^3 p$ is
the unique best approximation of $b_\lambda$. Furthermore, for the
rate of convergence we have the following two cases:
\begin{itemize}
    \item[a)] For $\lambda = \frac{1}{2}$ it holds $q=1$, i.e. the sequence $\left(\left|\tan\angle[\bar{v}, v_{k}]\right|\right)_{k \in \N}$ converges sublinearly.
    \item[b)] For $\lambda < \frac{1}{2}$ the ALS method converges Q-linearly with the convergence rate
        \begin{equation*}
            q_\lambda = \left[\frac{\lambda}{2} \left(3 \lambda + \lambda^2 + \sqrt{(3 \lambda + \lambda^2)^2 + 4 \lambda}\right)\right]^3.
        \end{equation*}
\end{itemize}
This example is not restricted to $d=3$. The extension to higher
dimensions is straightforward, see \cite{AramMikeALS1} for details.
\end{itemize}
\end{remark}
\section{Numerical Experiments}\label{sec:numericalExamples}
In this subsection, we observe the convergence behavior of the ALS
method by using data from interesting examples and more importantly
from real applications. In all cases, we focus particularly on the
convergence rate.

\subsection{Example 1}
We consider an example introduced by Mohlenkamp in \cite[Section
4.3.5]{Mohlenkamp2013}. Here we have $A= \id$ and
\begin{equation*}
    b = 2 \underbrace{\underbrace{\left(
            \begin{array}{c}
              1 \\
              0 \\
            \end{array}
          \right)}_{e_1:=} \ten \left(
            \begin{array}{c}
              1 \\
              0 \\
            \end{array}
          \right) \ten \left(
            \begin{array}{c}
              1 \\
              0 \\
            \end{array}
          \right)}_{b_1:=} + \underbrace{\underbrace{\left(
            \begin{array}{c}
              0 \\
              1 \\
            \end{array}
          \right)}_{e_2:=} \ten \left(
            \begin{array}{c}
              0 \\
              1 \\
            \end{array}
          \right) \ten \left(
            \begin{array}{c}
              0 \\
              1 \\
            \end{array}
          \right)}_{b_2:=},
\end{equation*}
see Eq. (\ref{equ:deff}). The tensor $b$ is orthogonally
decomposable. Although the example is rather simple, it is of great
theoretical interest. It follows from Theorem \ref{theorem:isoAP}
and \cite{AramMikeALS1} that the rate of convergence for an ALS
micro- step is
\begin{equation*}
    q_\mu= \limsup_{k \rightarrow
    \infty}\left| \frac{q_{k,\mu}^{(s)}}{q_{k,\mu}^{(c)}}\right| = 0.
\end{equation*}
Here the ALS method converges Q-superlinearly. Let $\tau\geq 0$, our
initial guess is defined by
\begin{equation*}
    v_0(\tau):= \left(
               \begin{array}{c}
                 \tau \\
                 1 \\
               \end{array}
             \right) \ten \left(
                            \begin{array}{c}
                              \tau \\
                              1 \\
                            \end{array}
                          \right) \ten \left(
                            \begin{array}{c}
                              \tau \\
                              1 \\
                            \end{array}
                          \right).
\end{equation*}
Since
\begin{equation*}
    4\dotprod{\left(\begin{array}{c}
                 1 \\
                 0 \\
               \end{array}
             \right)}{\left(\begin{array}{c}
                  \tau \\
                 1 \\
               \end{array}
             \right)}^2=4 \tau^2 \quad \mbox{and} \quad     \dotprod{\left(\begin{array}{c}
                 0 \\
                 1 \\
               \end{array}
             \right)}{\left(\begin{array}{c}
                 \tau \\
                 1 \\
               \end{array}
             \right)}^2=1,
\end{equation*}
we have for $\tau <\frac{1}{2}$ that the initial guess $v_0(\tau)$
dominates at $b_2$. Therefore, the ALS iteration converge to $b_2$,
see \cite{AramMikeALS1} for details. In the our numerical test, the
tangents of the angle between the current iteration point and the
corresponding parameter of the dominate term $b_l$ ($1\leq l \leq
2$) is plotted in Figure \ref{bild:Beispiel1}, i.e.
\begin{eqnarray}\label{eq:defTan}
  \tan \phi_{k, l}
  &=&\sqrt{\frac{1-\cos^2{\phi_{k,l}}}{\cos^2{\phi_{k,l}}}},
\end{eqnarray}
where $\cos{\phi_{k,l}} = \frac{\dotprod{p_1^k}{e_l}}{\|p_1^k\|}$.
\begin{figure}[h]\label{bild:Beispiel1}
  \centering
  {\begin{turn}{-90}  \image{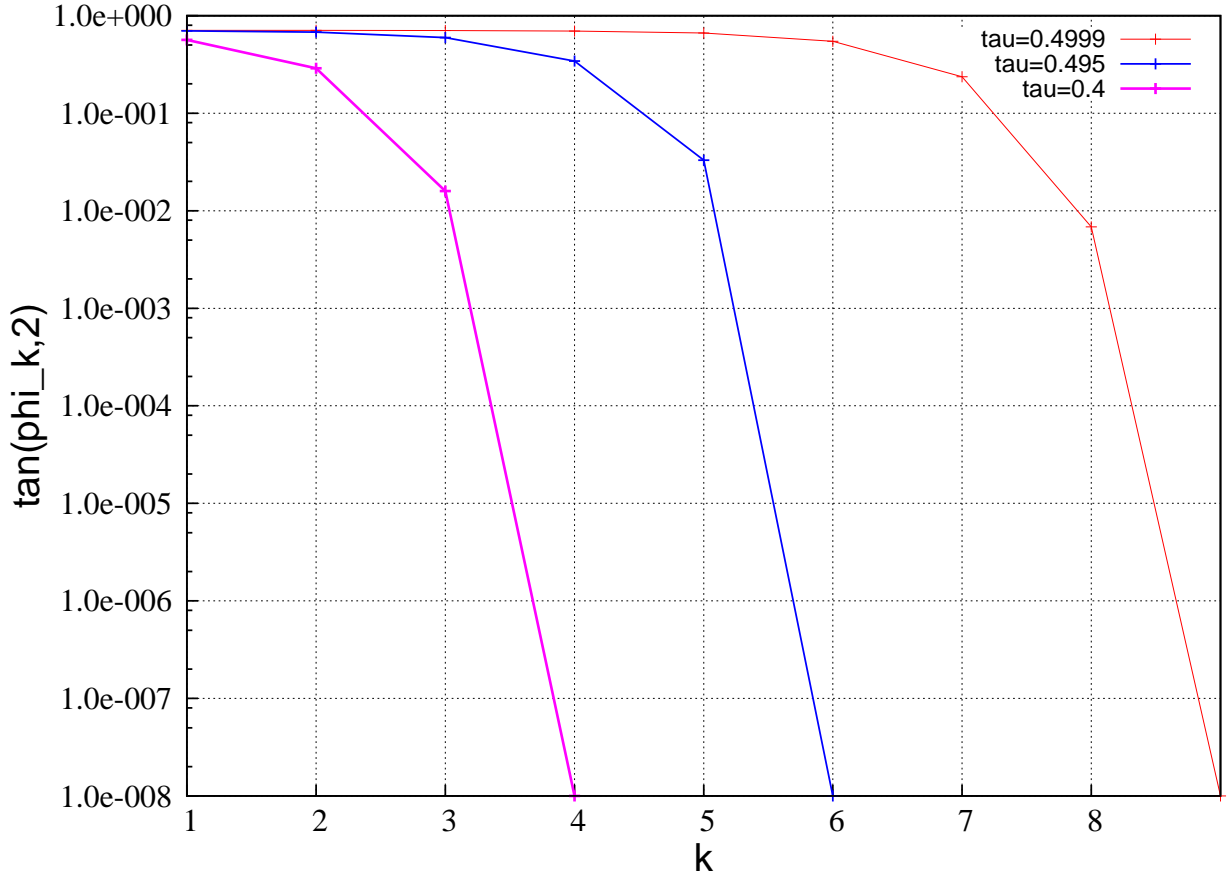}{0.7\textwidth} \end{turn}}
  \caption{The tangents $\tan \phi_{k, 2}$ from Eq. (\ref{eq:defTan}) is plotted for $\tau \in \{0.4, \, 0.495, \,0.4999\}$.}
  \label{bild:tan}
\end{figure}

\subsection{Example 2}
Most algorithms in ab initio electronic structure theory compute
quantities in terms of one- and two-electron integrals. In
\cite{esbe2011} we considered the low-rank approximation of the
two-electron integrals. In order to illustrate the convergence of
the ALS method on an example of practical interest, we use the
two-electron integrals of the so called AO basis for the CH$_4$
molecule. We refer the reader to \cite{esbe2011} for a detailed
description of our example. The ALS method converges here
Q-linearly, see Figure \ref{bild:CH4}.

\begin{figure}[h]
  \centering
  {\begin{turn}{-90}  \image{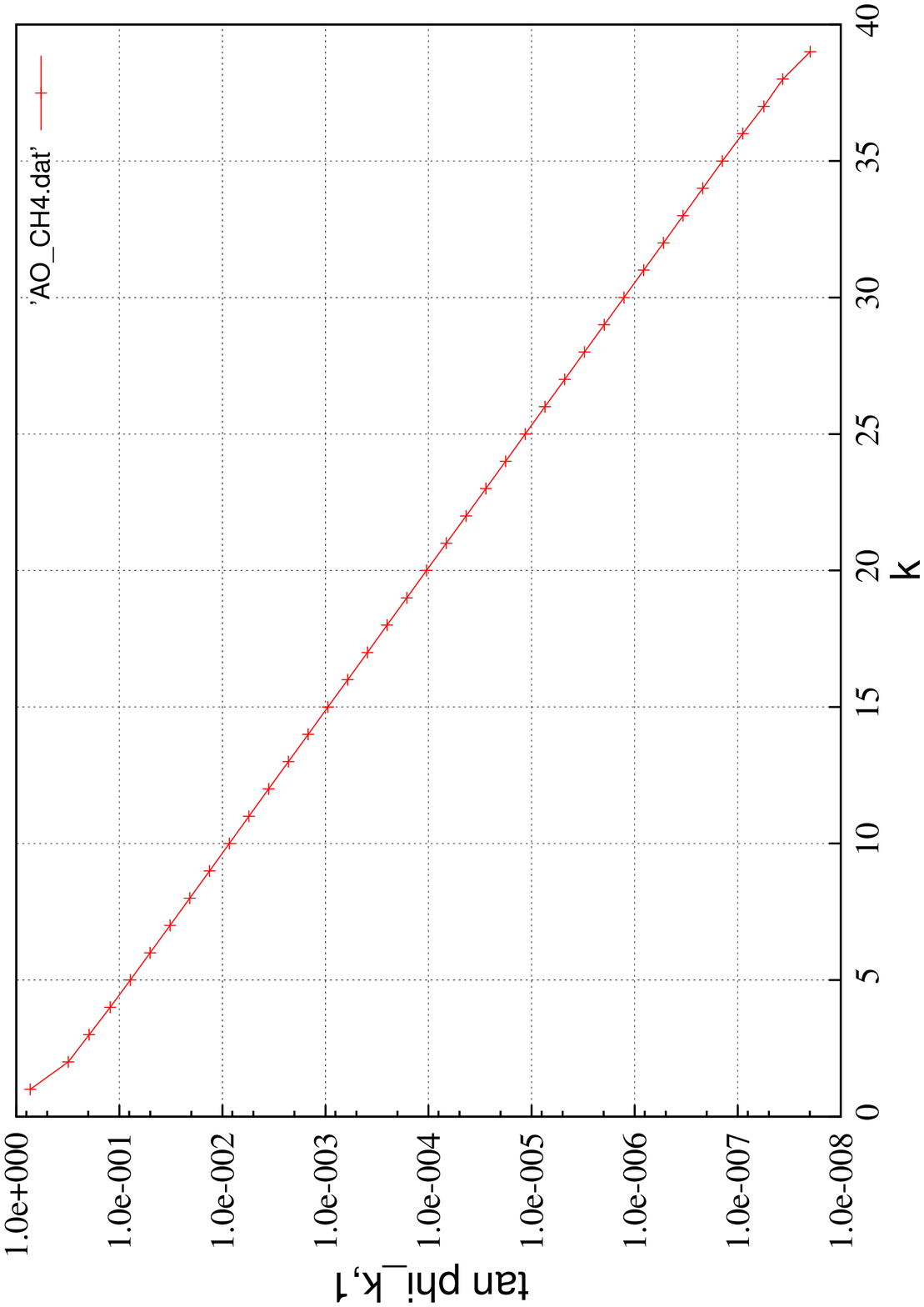}{0.7\textwidth} \end{turn}}
  \caption{The approximation of two-electron integrals for methane is considered. The tangents of the angle between the current iteration point and the
limit point with respect to the iteration number is shown.}
  \label{bild:CH4}
\end{figure}

\subsection{Example 3}
We consider the tensor
\begin{equation*}
    b_\lambda = \Ten_{\mu=1}^3 p + \lambda \left(p \ten q \ten q + q \ten p \ten q + q \ten q \ten p \right)
\end{equation*}
from Remark \ref{rem:convergenceRate}. The vectors $p$ and $q$ are
arbitrarily generated orthogonal vectors with norm $1$. The values
of $\tan (\phi_{1, k})$ are plotted, where $\phi_{1,k}$ is the angle
between $p_1^k$ and the limit point $p$. For the case $\lambda =
0.5$ the convergence is sublinearly, whereas for $\lambda < 0.5$ it
is Q-linearly. According to Theorem \ref{theorem:isoAP} and
\cite{AramMikeALS1}, the rate of convergence for an ALS micro-step
is given by
\begin{equation*}
    q_\lambda= \limsup_{k \rightarrow
    \infty}\left| \frac{q_{k,1}^{(s, \lambda)}}{q_{k,1}^{(c, \lambda)}}\right| = \frac{\lambda}{2} \left(3 \lambda + \lambda^2 +
\sqrt{(3 \lambda + \lambda^2)^2 + 4 \lambda}\right).
\end{equation*}
For $\lambda=0.46$, we have for the convergence rate $q_{0.46}=
0.847$. In Figure \ref{bild:rateOfConvergence} the ratio $\frac{\tan
(\phi_{1, k+1})}{\tan (\phi_{1, k})}$ is plotted. The ratio
$\frac{\tan (\phi_{1, k+1})}{\tan (\phi_{1, k})}$ perfectly matches
to $q_{0.46}=0.847$. This plot shows on an example the precise
analytical description of the convergence rate.

\begin{figure}[h]
  \centering
  \subfigure[The tangents $\tan \phi_{k, 1}$ for $\lambda \in \{0.1, \, 0.2, \,0.3 \}$.]{{ \begin{turn}{-90}  \image{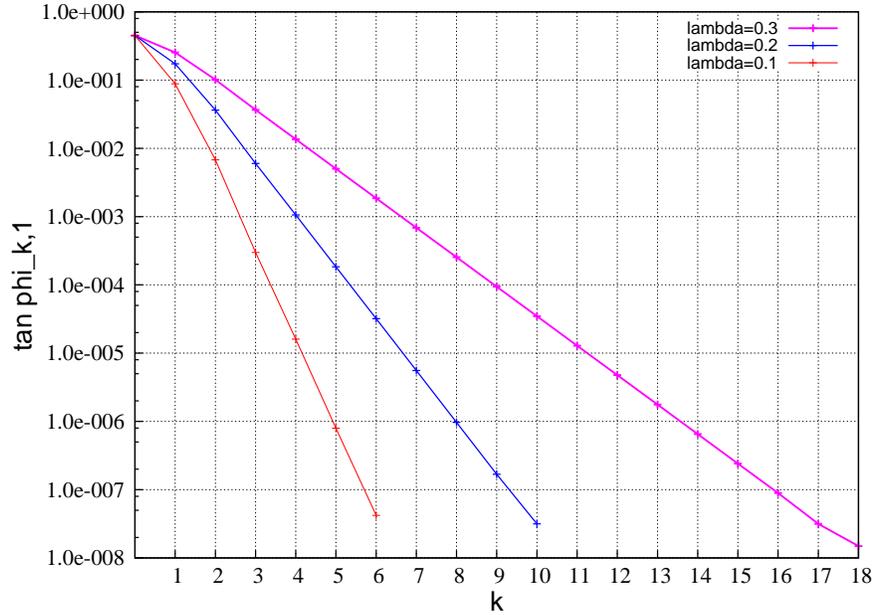}{0.49\textwidth} \end{turn}}}\label{bild:data02}\hfill
  \subfigure[The tangents $\tan \phi_{k, 1}$ for $\lambda \in \{0.44, \, 0.46, \,0.48 \}$.]{{ \begin{turn}{-90}\image{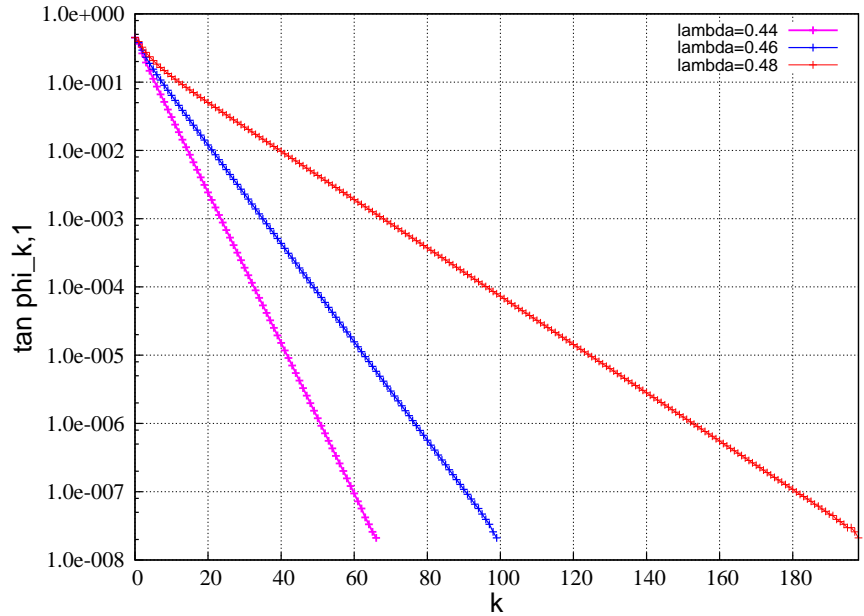}{0.49\textwidth}   \end{turn}}}\label{bild:data05}
  \caption{The approximation of $b_\lambda$ from Remark \ref{rem:convergenceRate} is considered. The tangents of the angle between the current iteration point and the
limit point with respect to the iteration number is plotted. For
$\lambda < 0.5$ the sequence converges Q-linearly with a convergence
rate $q_\lambda = \frac{\lambda}{2} \left(3 \lambda + \lambda^2 +
\sqrt{(3 \lambda + \lambda^2)^2 + 4 \lambda}\right)<1$.}
\end{figure}

\begin{figure}[h]
  \centering
 \begin{turn}{-90}\image{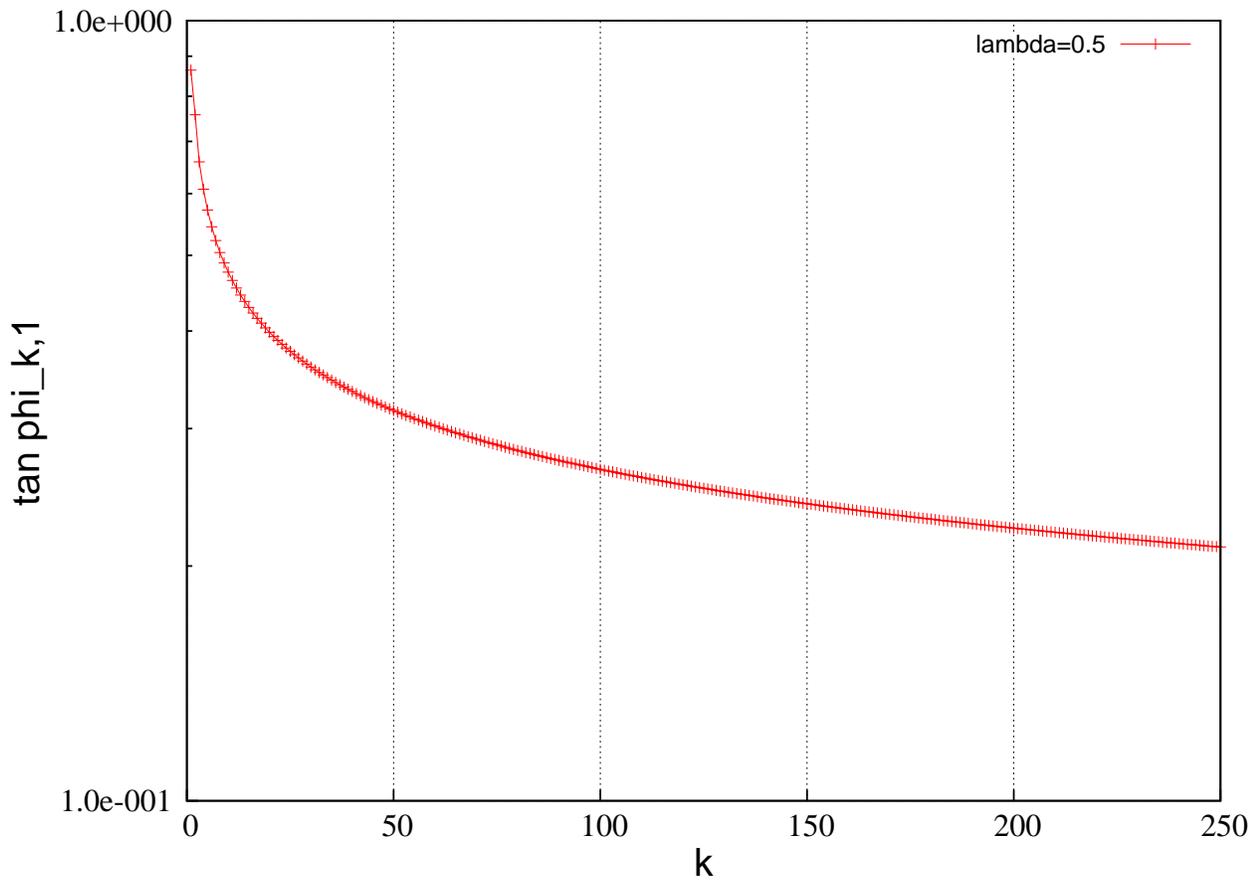}{0.7\textwidth}   \end{turn}\label{bild:data05}
  \caption{The approximation of $b_\lambda$ from Remark \ref{rem:convergenceRate}  is considered. The tangents of the angle between the current iteration point and the
limit point with respect to the iteration number is plotted. For
$\lambda=0.5$, we have sublinear convergence since $q_{0.5}=1$.}
\end{figure}

\begin{figure}[h]
  \centering
 \begin{turn}{-90}\image{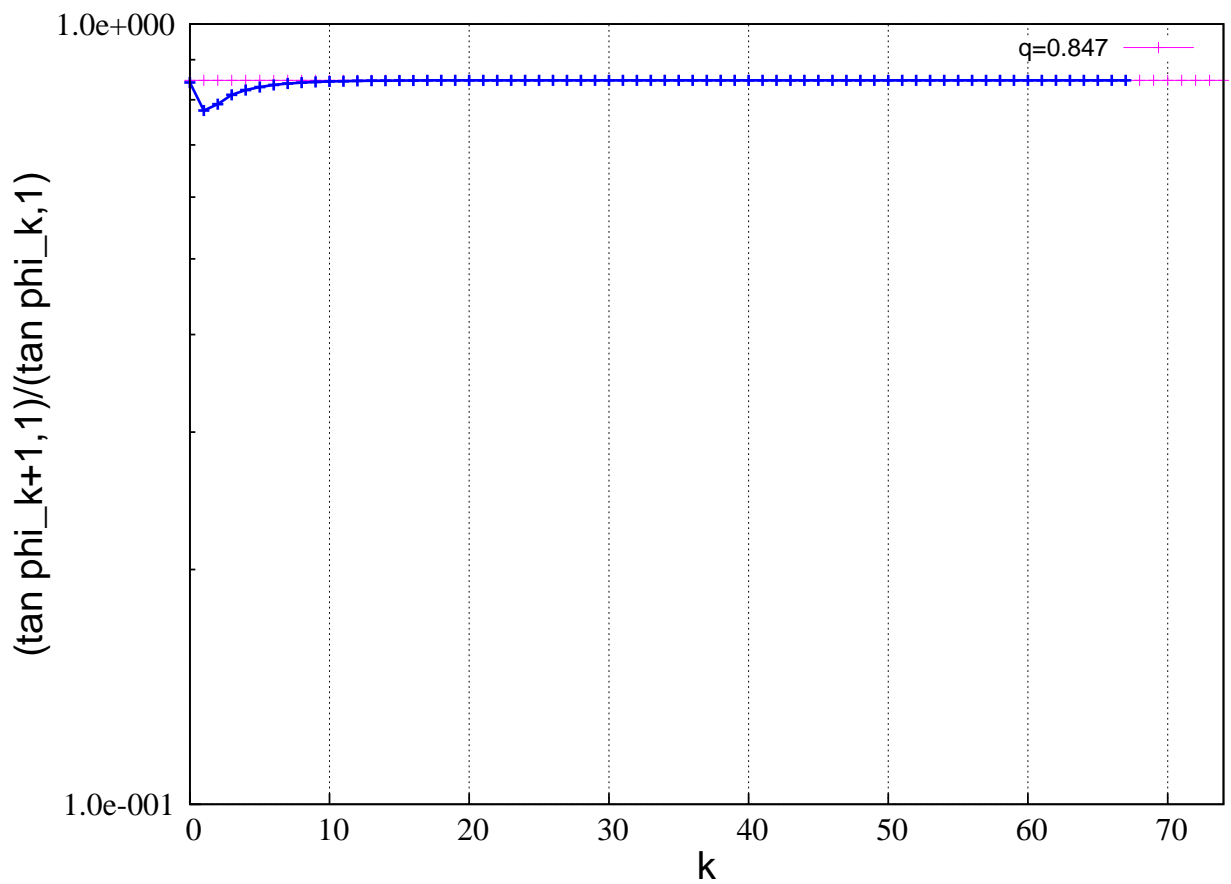}{0.7\textwidth}   \end{turn}\label{bild:rateOfConvergence}
  \caption{The ratio $\frac{\tan (\phi_{1,k+1})}{\tan (\phi_{1, k})}$ is plotted for
  $\lambda=0.46$. The rate of convergence from Theorem \ref{theorem:isoAP}
  is for this example equal to $0.847$. The plot illustrates that the description of the convergence rate is accurate and sharp.}
\end{figure}


%

\bibliographystyle{plain}
\bibliography{../../BibTeX/all}

\end{document}